\documentclass[12pt,reqno]{amsart}
\usepackage[usenames]{color}
\usepackage{tikz}
\usepackage{amsmath,verbatim,graphicx,epstopdf,enumerate}
\usepackage{amssymb,mathtools, mathrsfs} 

\usepackage[usenames]{color}
\usepackage{soul}
\newcommand{\mb}{\mathbb}
\newcommand{\mc}{\mathcal}

\newcommand{\f}{\frac}

\newcommand{\rt}{\sqrt}

\newcommand{\dd}{\partial}

\newcommand{\nn}{\nonumber}
\newcommand{\ep}{\epsilon}
\newcommand{\qd}{\quad}

\newcommand{\iy}{\infty}

\newcommand{\tn}{\textnormal}

\usepackage[pagewise]{lineno}
\newtheorem{theorem}{Theorem}[section]
\newtheorem{lemma}[theorem]{Lemma}
\newtheorem{proposition}[theorem]{Proposition}

\newtheorem{defn}[theorem]{Definition}

\numberwithin{equation}{section}

\newtheorem{prop}[theorem]{Proposition}

\newcommand{\la}{\mathcal{L}}

\newcommand{\real}{\mathbb{R}}

\newcommand{\bp}{\begin{prob}}
	\newcommand{\bpr}{\begin{proof}}
		\newcommand{\epr}{\end{proof}}

	\renewcommand{\d}{\mathrm{d}}

	\newcommand{\Lc}{\mathcal{L}}

	\newcommand{\Rb}{\mathbb{R}}

	\usepackage{amsmath}
	\theoremstyle{definition}
	
	\newcommand{\beqn}{\begin{equation}}
    \newcommand{\eeqn}{\end{equation}}
	\def\Bbar{{\overline{B}}}
	\def\aacu{{\acute{a}}}
	\def\bacu{{\acute{b}}}
	\def\cacu{{\acute{c}}}
	\def\qacu{{\acute{q}}}
	\def\uacu{{\acute{u}}}
	\def\vacu{{\acute{v}}}
	\def\wacu{{\acute{w}}}
	\def\alphaacu{{\acute{\alpha}}}
	\def\psiacu{{ \acute{\psi}}}
	\def\phiacu{{ \acute{\phi}}}
	\def\cleq{{\preccurlyeq}}
	\def\cgeq{{\succeq}}
	\def\Fcal{{\mc{F}}}
	\def\Tcal{{\mc{T}}}
	\def\Lcal{{\mc{L}}}
	\def\Lacu{{\acute{\Lcal}}}
	\def\Mcal{{\mc{M}}}
	\def\abar{{\overline{a}}}
	\def\bbar{{\overline{b}}}
	\def\cbar{{\overline{c}}}
	\def\qbar{{\overline{q}}}
	\def\ubar{{\overline{u}}}
	\def\vbar{{\overline{v}}}
	
	\def\smeq{{{\small =}}}
	\def\Dbar{{\overline{D}}}
	\def\pa{{\partial}}
	\def\la{\langle}
\def\ra{\rangle}
\def\ep{\epsilon}
\def\nn{\nonumber}
\def\evansibvp{\eqref{eq:evansde} - \eqref{eq:evansbc} ~}
\DeclareMathOperator*{\esssup}{ess\,sup} 

\begin{document}
	\title[Lipschitz stability for a hyperbolic inverse problem] {Point sources and stability for an inverse problem for a hyperbolic PDE with space and time dependent coefficients}
	
	\author[Krishnan, Rakesh and Senapati] {Venkateswaran P. Krishnan$^{*}$, Rakesh$^{\dag}$, and Soumen Senapati$^{\ddag}$}
	\address{$^{*}$ TIFR Centre for Applicable Mathematics, Bangalore 560065, India. 
		\newline\indent E-mail:{\tt \ vkrishnan@tifrbng.res.in}}
	\address{$^{\dag}$	Department of Mathematical Sciences, University of Delaware, Newark, DE 19716, USA.
	\newline\indent E-mail:{\tt \ rakesh@udel.edu}}
	\address{$^{\ddag}$ TIFR Centre for Applicable Mathematics, Bangalore 560065, India. 
		\newline\indent E-mail:{\tt \ soumen@tifrbng.res.in}}

 	\begin{abstract}
		We study stability aspects for the determination of space 
		and time-dependent lower order perturbations of the wave 
		operator in three space dimensions with point sources. The 
		problems under consideration here are formally determined and 
		we establish Lipschitz stability results for these problems. 
		The main tool in our analysis is a modified version of 
		Bukghe\u{\i}m-Klibanov method based on Carleman estimates.
    \end{abstract}
 	\maketitle

		\vspace{2mm}
		

\section{Introduction and main results}\label{Introduction}	

Throughout this article, all functions are real valued, $T$ denotes a positive real number, $B$ denotes the origin centered open unit ball in $\real^d$ for any positive integer $d$ and, for $\rho>0$, $\rho B$ is the origin centered open ball of radius $\rho$ in $\real^d$.

 For functions $a(x,t), c(x,t)$ and the vector field 
 $b(x,t) = (b^1(x,t), b^2(x,t), b^3(x,t))$ on $\real^3 \times \real$, define the hyperbolic operator
 		\begin{equation}\label{HyperbolicPDE}
 		\begin{aligned}
 		\mc{L}_{a,b,c}:= (\dd_t - a)^2-(\nabla - b)^2 + c
 		= \square - 2 a \dd_t + 2 b \cdot \nabla + q
 		\end{aligned}
 		\end{equation}
 		where 
 		\begin{align*}
 		    q = c - a_t + \nabla \cdot b + a^2 - |b|^2.
 		\end{align*}
 To avoid introducing too many symbols, we use $\mc{L}_{a, b, c}$ and $\Lc_{a,b,q}$ interchangeably since the form of the operator will be clear from the context.

 		Suppose $a(x, t), c(x,t)$ and $b(x,t)$ are smooth compactly supported functions and a vector field on $\mb{R}^3 \times \mb{R}$ with support in $\Bbar \times \real$. 
 		Given $\xi \in \mb{R}^3 \setminus \overline B$, $\tau \in \real$, let $U(x, t; \xi, \tau)$ be the solution of the IVP
 		\begin{align}
 		\label{IVP for Heaviside 1}  \mc{L}_{a, b, c} U(x,t;\xi,\tau) = 4 \pi H(t - \tau) \delta(x - \xi), & \qd \tn{in } \mb{R}^3 \times \mb{R}, \\
 		 \label{IVP for Heaviside 2}     U(x, t; \xi, \tau) = 0, & \qd \text{for } x \in \real^3, ~ { t < \tau},
 		\end{align}
 		 where $H$ is Heaviside function, and let $V(x, t; \xi, \tau)$ be the solution of the IVP
 		\begin{align}
 		 \label{IVP for delta 1}    \mc{L}_{a, b, c} V(x,t;\xi,\tau) = 4 \pi \delta(x - \xi, t - \tau), 
 		 & \qd \tn{in } \mb{R}^3 \times \mb{R}, \\
 		 \label{IVP for delta 2}    V(x, t; \xi, \tau) = 0, & \qd 
 		 \text{for } x \in \real^3, ~ { t < \tau}.
 		\end{align}
 		Define the forward map
 		\begin{align}\label{forward map}
 		    \mc{F} : (a, b, c) \to \left.[U, U_t, V, V_t] (x, T; \xi, \tau)\right|_{x \in \mb{R}^3,\xi \in E,\tau \in (-\iy, T] }
 		\end{align}
 		which measures the medium response at the final time $t=T$, to waves generated by a
 		point source at $\xi$ in a {\bf finite} subset $E$ of $\real^3$, with sources activated at times $\tau$ varying over the interval
 		 $(-\infty, T]$. Here $(a, b, c)$ represents the medium properties and the medium is uniform outside the cylinder $\overline B \times [0, T]$. Our goal is to study the injectivity and stability of $\mc{F}$. The problem is formally determined in the sense that the data set depends on four real parameters - three for the receiver locations
 		 $(x,t \smeq T) \in \real^3 \times \{t \smeq T\}$ and one for the time delay $\tau \in \real$ - while the unknown coefficients $(a, b, c)$ are also functions of the 
 		 four variables $(x,t) \in \real^3 \times \real$.

 		 		This work is a follow up of our previous work \cite{KRS1}, in which we derived Lipschitz stability estimates for the determination of the coefficients $a$, $b$ (up to a gauge term) and $c$ in \eqref{HyperbolicPDE} in space dimensions $d\geq 2$, where we used plane wave sources; our current work uses point sources. Our current results are only for the three dimensional case ($d=3$) since the ansatz for the fundamental solution of $\Lc_{a,b,c}$
 		 		becomes unwieldy for $d \neq 3,1$.
 		 		
 		 		Similar to the work \cite{KRS1}, we derive uniqueness and Lipschitz stability estimates for the recovery of time-dependent coefficients $a,b$ (up to a gauge term) and $c$ for a formally determined inverse problem with point sources.    Our proofs are based on suitable modifications of the ideas of Bukhge\u{i}m and Klibanov \cite{Bukhgeuim_Klibanov_Uniqueness_1981} which were based on Carleman estimates. A part of our current work as well as \cite{KRS1} used ideas from \cite{MPS2020} about extending results for the $a=0, b=0$, arbitrary $c$ case to the general $a,b,c$ case. The work \cite{MPS2020} deals with the recovery of time-independent first order coefficients of a hyperbolic PDE in a formally determined set-up as well, and extends the ideas in \cite{RS_1} 
 		 		to the general $a,b,c$ case.
 		 
 		  A point source inverse back-scattering problem in $\Rb^3$ involving the recovery of time-independent potential, with data coming from coincident source-receiver pairs varying over the surface of a sphere, was considered in \cite{RU2}. They showed the unique recovery of angularly controlled potentials, in particular, radial potentials, from such formally determined data. This was further investigated in \cite{Blasten}, where a logarithmic stability estimate for the recovery of time-independent angularly controlled potentials for the point source inverse backscattering problem was shown.

  Our earlier work \cite{KRS1} has a detailed survey of the literature on hyperbolic inverse problems for the operator $\Lc_{a,b,c}$ (and some work even for non-constant velocity) with time-independent/time-dependent coefficients where the data is either measured on the lateral boundary and on $t=T$, or the data is measured only on the lateral boundary or a part of it, or the sources are located only on the lateral boundary but not in the initial data, or the data is the far-field pattern in the frequency domain. For this reason, we do not repeat the literature review here.

 		Before discussing the main results of the article, we introduce some 
 		definitions and notation. Given $\xi \in \mb{R}^3 \setminus \overline B$ and $\tau \in \mb{R}$, 
 		define the conical region 
 		\begin{align*}
 		    Q_{\xi, \tau} = \{(x,t) \in \mb{R}^3 \times \mb{R} ;\ |x- \xi| + \tau \le t\le T\}
 		\end{align*}
 		and denote its top (horizontal) and conical boundaries by
 		\begin{align*}
 		    H_{\xi,\tau}=Q_{\xi,\tau}\cap \{t=T\}, \qquad  C_{\xi,\tau}=Q_{\xi,\tau}\cap \{t=\tau+|x-\xi|\},
 		\end{align*}
 		respectively.
 		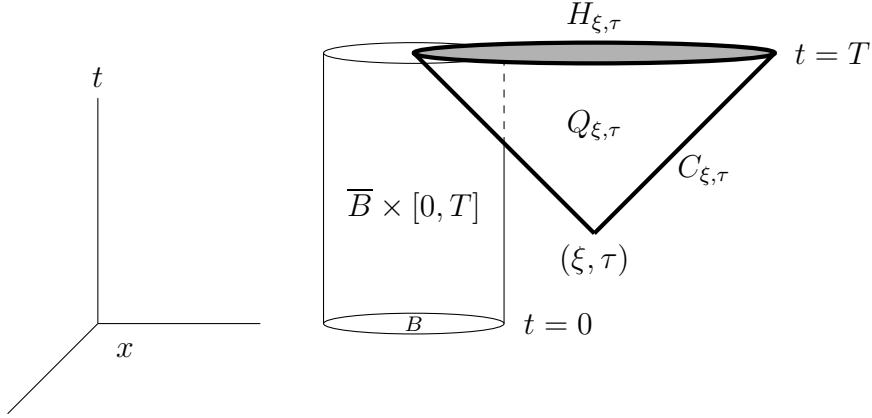
\begin{figure}[h]
 		\centering
 		\begin{tikzpicture} [scale = 1.2]
        \draw (0,0)--(1.8,0);
        \draw (0,0)--(0,2.5);
        \draw (0,0)--(-1,-1);
        \node [above] at (0,2.5) {$t$};
        \node at (0.3,-0.3) {$x$};
        \draw (3.5,3) ellipse (1 and 0.115);
        \draw (3.5,0) ellipse (1 and 0.115);
        \draw (2.5,0)--(2.5,3);
        \draw (4.5,0)--(4.5,2);
        \draw [dashed] (4.5,1.98)--(4.5,3);
        \node at (3.5,0) {\tiny{$B$}};
        \node at (3.5,1.3) {$\overline B\times [0,T]$};
        \draw (5.5,3) ellipse (2 and 0.115);
        \filldraw [ultra thick][fill=gray!60!white, draw=black] (5.5,3) ellipse (2 and 0.115);
        \draw [ultra thick] (5.5,1)--(3.5,3);
        \draw [ultra thick] (5.5,1)--(7.5,3);
        \node [below] at (5.5,1) { $(\xi, \tau)$ };
        \node at (5.5,2.2) {$Q_{\xi, \tau}$};
        \node [right] at (4.6,0) { $t=0$ };
        \node [right] at (7.6,3) { $t=T$ }; 
        \node [right] at (6.3,1.7) { $C_{\xi,\tau}$ };
        \node [above] at (5.5,3.1) {$H_{\xi,\tau}$};
 		\end{tikzpicture}
 		\caption{The conical domain $Q_{\xi, \tau}$ and its boundaries}
 		\end{figure}

 		Given $\sigma >0$, $M$ a submanifold of $\mb{R}^3\times\mb{R}$, and a function
 		$f: M \to \real$, define the weighted norms
 		\begin{align*}
 		    \|f\|_{0,M,\sigma}=\left(\int_{M}e^{2\sigma t}|f|^2\right)^{1/2}, \qd \|f\|_{1,M,\sigma}=\left(\int_{M}e^{2\sigma t}\left(|\nabla_{M}f|^2+\sigma^2|f|^2\right)\right)^{1/2}
 		\end{align*}
 		where $\nabla_{M}$ consists of the first order derivatives in directions tangential to $M$. For $x, \xi \in \real^3$, $x \neq \xi$, define
 		\begin{align}\label{defn of theta}
 		    r=|x-\xi|,\ \theta = \f{x-\xi}{|x-\xi|},\qd \dd_r = \theta\cdot\nabla. 
 		\end{align}
 
 		For a compactly supported smooth function $a$ and vector field $b$ on $\real^3 \times \real$, and 
 		$\xi \in \real^3$ such that $\{\xi\} \times \real$ is disjoint from the supports of $a,b$, define 
 		\begin{align}
 		    \alpha(x,t;\xi) =\f{1}{r}\exp\left(\int_{0}^{r}(a+\theta\cdot b)(x-s\theta,t-s)\ \d s\right), \qd x\neq \xi.
 		    \label{eq:alphadef}
 		\end{align}
 		Note that $\alpha(x,t;\xi) = r^{-1}$ in a punctured cylindrical neighborhood of $\{\xi\} \times \real$ 
 		and $\alpha$ satisfies the equivalent transport equations
 		\begin{align}\label{transport eqn}
 		   \left( \dd_t + \theta \cdot \nabla + r^{-1} \right) \alpha = ( a + \theta \cdot b ) \alpha,
 		   \qquad
 		   (\pa_t + \pa_r - (a+ \theta \cdot b))(r \alpha) =0,
 		   \qquad x \neq \xi.
 		\end{align}
 		This follows from the identity
 		\begin{align}\nn
 		  r \left( \dd_t + \theta \cdot \nabla + r^{-1} \right) \alpha=  \left( \dd_t + \dd_r \right) (r \alpha)
 		\end{align}
 		and that
 		\begin{align*}
 		& (r \alpha)^{-1} (\pa_t + \pa_r)(r \alpha)
 		\\
 		 & = \exp \left( - \int_{0}^{r} ( a + \theta \cdot b ) ( x - s \theta ,t - s)\ \d s \right) \left( \dd_t + \dd_r \right) \exp \left( \int_{0}^{r} ( a + \theta \cdot b ) ( x - s \theta ,t - s)\ \d s \right)  \\
 		 & =  \int_{0}^{r} \left( a_t + \theta \cdot b_t + a_r + \theta \cdot b_r  \right) ( x - s \theta, t - s)\ \d s +(a+\theta \cdot b) (x-r\theta, t-r)\\
 		 & = - \int_{0}^{r} \f{\d} {\d s} ( a + \theta \cdot b ) ( x - s \theta, t - s)\ \d s +(a+\theta \cdot b) (x-r\theta, t-r). \\
 		 & = ( a + \theta \cdot b ) (x, t).
 		\end{align*}
 		We also define the useful first order operators
 		\begin{align*}
 		\mc{M} = -2a\dd_t + 2b\cdot\nabla + q , \qd \Tcal = \dd_t + \theta\cdot\nabla - (a+\theta\cdot b) + r^{-1},
 		\qquad x \neq \xi;
 		\end{align*}
 		note that $\mc{M}$ is zero in a punctured cylindrical neighborhood of $\{\xi\} \times \real$ and
 		\eqref{transport eqn} may be rewritten as
 		\beqn
 		\Tcal \alpha =0, \qquad x \neq \xi.
 		\label{eq:Talpha}
 		\eeqn

 		We first address the structure of $U$, $V$ and the well-posedness of the IVPs defining $U$, $V$.
 		\begin{proposition}\label{Heaviside}
 		If $a, c,$ and $b$ are compactly supported smooth functions and a vector field on 
 		$\real^3 \times \mb{R}$, respectively and $\{\xi\} \times \Rb$ is disjoint from the support of $a,b,c$, then 
 		the IVP 
 		\begin{align}
 		 \label{IVP for Heaviside 1A} \mc{L}_{a, b, c} U(x,t;\xi,\tau) = 4 \pi H(t - \tau) \delta(x - \xi), & \qd \tn{in } \mb{R}^3 \times \mb{R}, \\
 		 \label{IVP for Heaviside 2B}     U(x,t;\xi,\tau) = 0, & \qd \text{for } x \in \real^3, ~ { t < \tau},
 		\end{align}
 		admits a unique distributional solution $U(x, t; \xi, \tau)$. Further,
 		\begin{align*}
 		    U(x, t ; \xi, \tau) = \f{ H (t- \tau - |x-\xi|) } { |x - \xi| } + u (x,t;\xi,\tau) H (t - \tau - |x-\xi|),
 		\end{align*}
 		where $u(x,t;\xi,\tau)$ is a smooth function in the region $\{(x,t) \in \mb{R}^3 \times \mb{R};\ t\ge\tau+|x-\xi|\}$ and is a smooth solution of the characteristic BVP
 		\begin{align}
 		\label{ch BVP for Heaviside 1}   \mc{L}_{a,b,c} u = -\mc{M} \left( |x-\xi|^{-1} \right), & \qd  t > \tau + |x-\xi|,\\
 		\label{ch BVP for Heaviside 2}   u (x,t;\xi,\tau) = \alpha(x,t;\xi) - |x-\xi|^{-1},  & \qd  t = \tau + |x-\xi|,\ x \neq \xi.
 		\end{align}
 		Finally, if the compactly supported coefficients $a,b,c$ satisfy  $\|[a, b, c]\|_{C^{20} (\Rb^3\times \Rb)} \le M$, then
 		\begin{align*}
 		    \|u\|_{C^3 (Q_{\xi, \tau})} \le C
 		\end{align*}
 		where $C$ depends only on $T$, $M$ and the reciprocal of the distance of $\{\xi\} \times \real$ from the 
 		support of $a,b,c$.
 		\end{proposition}

 		\begin{proposition}\label{Delta}
 		If $a,c$ and  $b$ are compactly supported smooth functions and a vector field on
 		$\real^3 \times \real$, respectively, and $\{ \xi \} \times \real$ is disjoint from the support of $a,b,c$, the IVP 
 		\begin{align}
 		 \label{IVP for delta 1A}   { \mc{L}_{a, b, q} V(x,t;\xi,\tau) = 4 \pi \delta(x - \xi, t - \tau)}, 
 		 & \qd \tn{in } \mb{R}^3 \times \mb{R}, \\
 		 \label{IVP for delta 2B}    V(x,t;\xi,\tau) = 0, & \qd 
 		 \text{for } x \in \real^3, ~ {t < \tau}.
 		\end{align}
 		admits a unique distributional solution $V(x, t; \xi, \tau)$ of the 
 		form
 		\begin{align*}
 		    V(x, t; \xi, \tau) = \alpha(x, t; \xi) \delta(t - \tau - |x - \xi|) +v(x, t; \xi, \tau)H(t - \tau - |x - \xi|)
 		\end{align*}
 		where $v(x,t;\xi,\tau)$ is a smooth function in the region $\{(x,t)\in\mb{R}^3\times\mb{R};\ t\ge\tau+|x-\xi|\}$ and solves the characteristic BVP 
 		\begin{align} 
 		 \label{ch BVP for delta 1}  \mc{L}_{a,b,c} v = 0, & \qd  t > \tau + | x - \xi|,\\
 		 \label{ch BVP for delta 2} 
 		 \Tcal v= - \f{1}{2} \mc{L}_{a,b,c} \alpha, & \qd  t = \tau + |x-\xi|,\ x \neq \xi.
 		\end{align}
 		Finally, if the compactly supported coefficients $a,b,c$ satisfy $\|[a, b, c]\|_{C^{22} (\Rb^3 \times \Rb)} \le M$, then
 		\begin{align*}
 		    \|v\|_{C^3 (Q_{\xi, \tau})} \le C
 		\end{align*}
 		where $C$ depends only on $T,M$ and the reciprocal of the distance
 		of $\{\xi\} \times \real$ from the support $a,b,c$
 		\end{proposition}

 		For future use we make several observations about $u$ and $v$.
 		\begin{itemize}
 		\item Since $a,b,c$ are supported away from $(x \smeq \xi, t \smeq \tau)$, in some neighborhood 
 		of $(x \smeq \xi, t \smeq \tau)$ we have
 		\begin{gather*}
 		 \mc{M}=0, ~~~ \alpha(x,t;\xi)= |x - \xi|^{-1}, ~~~ 
 		 \mc{L}_{a,b,c} \alpha = - \Delta (|x-\xi|^{-1})= 0, \qquad \text{for } x \neq \xi,
 		\\
 		U(x,t;\xi,\tau) = \frac{H(t-\tau-|x-\xi|)}{|x-\xi|}, \qquad
 		V(x,t;\xi,\tau) = \frac{\delta(t-\tau-|x-\xi|)}{|x-\xi|}.
 		\end{gather*}
 	    Hence $u\smeq 0$ and $v \smeq 0$ in a neighborhood of $(x \smeq \xi, t \smeq \tau)$ and the 
 		singular terms $|x-\xi|^{-1}$ in (\ref{ch BVP for Heaviside 1}),
 		(\ref{ch BVP for Heaviside 2}),  (\ref{ch BVP for delta 1}), 
 		(\ref{ch BVP for delta 2}) will never be an issue.
 		\item Suppose $a,b,c$ are supported in $\Bbar \times [0,T]$.
 		We claim that for $\tau> T+1-|\xi|$ the values of $u,v$ and their derivatives are zero on $t=T$. This is so because, for $\tau> T+1-|\xi|$, we have
 		\[
       	U(x,t;\xi, \tau) = \frac{H(t-\tau - |x-\xi|)}{|x-\xi|}, \qquad
 	    V(x,t;\xi, \tau) = \frac{\delta(t-\tau - |x-\xi|)}{|x-\xi|},
     	\]
        which may be readily verified because $\mc{L}_{a,b,c} = \Box$ on the supports of the right hand sides of the two expressions. 
        \item Suppose $a,b,c$ are supported in $\Bbar \times [0,T]$.
        If $\tau_1 < \tau_2 < - (1+|\xi|)$ then the values of $[u,v](\cdot, \cdot, \xi, \tau_1)$ and $[u,v](\cdot, \cdot, \xi, \tau_2)$ and their derivatives on $t=T$ are the same. This is so because, for $\tau_1 < \tau_2 < - (1+|\xi|)$, we have
 	    \begin{align*}
    	U(x,t;\xi, \tau_1) - U(x,t;\xi, \tau_2) & = \frac{H(t-\tau_1 - |x-\xi|)}{|x-\xi|} - 
    	\frac{H(t-\tau_2 - |x-\xi|)}{|x-\xi|}
    	\\
    	V(x,t;\xi, \tau_1) - V(x,t;\xi, \tau_2) & = \frac{\delta(t-\tau_1 - |x-\xi|)}{|x-\xi|} - 
    	\frac{\delta(t-\tau_2 - |x-\xi|)}{|x-\xi|};
    	\end{align*}
    	this may be readily verified because $\mc{L}_{a,b,c} = \Box$ on the supports of the right hand 
    	sides of the two expressions.
    	\item The previous two observations show that there is no new information about $a,b,c$
    	in the values of $u,v$ and their derivatives on $t=T$ for $\tau$ outside the interval
    	$[-(1+\xi|), T+1 - |\xi|]$.
    	{\item The relations between $a,b,c$ and the traces of $u$ and $v$ on the conical surface $C_{\xi,\tau}$ are what 
	makes the proofs possible. The trace of $u$ on $C_{\xi,\tau}$ depends only on $a,b$ so for the recovery of
	$a,b$ for a given $c$, our data is the trace of $u$ on $t=T$, as in Theorem \ref{thm:abstab}. 
	A tangential derivative of the trace of $v$ on $C_{\xi,\tau}$ equals a known positive multiple of $q$, if $a,b$ are known, so for the 
	recovery of $c$ or $q$, given $a,b$, our data is the trace of $v$ on $t=T$, as in Theorem \ref{thm:qstab}.
	For the recovery of $a,b,c$, our data consists of the traces of both $u$ and $v$ on $t=T$, as in
	Theorems \ref{thm:abcunique}, \ref{thm:abcstab}.}
	{\item From the uniqueness of solutions of initial value problems for hyperbolic PDEs, one can see that 
	$V(\cdot, \cdot, \cdot, \tau) = - U_\tau(\cdot, \cdot, \cdot, \tau)$, so there must be a relation between $u$ and $v$. We have not explored this question - note $u,v$ are defined on regions in $x,t$ space which depend on $\tau$ (and $\xi$). }
 		\end{itemize}

 	Now we describe the main results in our article. 
 	Our first result is about the stability for the problem of recovering $q$ from the data generated by a point source at a fixed location in space but activated at different times $\tau$.
 		\begin{theorem} [Stability for $q$]\label{thm:qstab}
 		Let $a, b$ be a smooth function and a smooth vector field on $\mb{R}^3 \times \real$ with support in $\Bbar \times [0,T]$ and 
 		$\xi \in \mb{R}^3 \setminus \overline{B}$. Given $M>0$, for all smooth functions 
 		$q, \qacu$ on $\real^3 \times \real$ with support in $\Bbar \times [0,T]$ and
 		$\| [q, \acute q, a, b] \|_{C^{21}(\overline{B}\times [0,T])}\le M$, we have
 		\begin{align*}
 		    \|q - \acute q\|_{0, \real^3 \times [0,T]} 
 		    \, \cleq 
 		    \int_{ -1 - |\xi|}^{T+1-|\xi| } \left( \| (v- \acute v) (\cdot, T; \xi, \tau) \|_{1,H_{\xi, \tau}} + \|(v_t - \acute v_t) (\cdot, T; \xi, \tau)\|_{0,H_{\xi, \tau}} \right) \d \tau.
 		\end{align*}
 		Here the constant is independent of $q, \qacu$, and
 		$v$, $\acute v$ are the functions associated to $(a,b,q)$ and $(a,b,\qacu)$ guaranteed by 
 		Proposition \ref{Delta}.
 		\end{theorem} 

 		The rest of our results pertain to the recovery of the vector field
 		$b$ and perhaps the functions $a,c$. For such results, we 
 		need sources at $4$ locations diverse enough to generate data to 
 		separate $a,b$.
 		\begin{defn}
 		Suppose $d$ is a positive integer and $D$ is a non-empty bounded open subset of $\real^d$. A set of locations 
 		$\xi_1, \cdots, \xi_{d+1}$ in $\real^d \setminus \Dbar$ is said to be {\bf diverse} with respect to $D$ if
 		\begin{equation}
 		\| [a,b]\| \leq C \| [ a+ \theta_1(x) \cdot b, \cdots, a + \theta_{d+1}(x) \cdot b]\|,
 		\qquad \forall x \in \Dbar, ~ \forall a\in \real, ~ \forall b \in \real^d,
 		\end{equation}
 		for some constant $C$ independent of $a,b,x$. Here $\| \cdot \|$ is the $l^2$ vector
 		norm in $\real^{d+1}$ and
 		\[
 		\theta_i(x) = \frac{x-\xi_i}{|x- \xi_i|}, \qquad x \in  \Dbar.
 		\]
 		\end{defn}
 		We do not have a characterization of all possible sets of locations diverse with respect to $D$ 
 		but Proposition \ref{prop:diverse} gives two ways to construct many such sets. 
 		A consequence of Proposition \ref{prop:diverse} (see the remark after Proposition \ref{prop:diverse}) is 
 		that if $\rho>0$ then
 		$Ne_1, Ne_2, Ne_3, N(e_1+e_2 + e_3)/3$ is a diverse set of locations with respect to $\rho B$ 
 		if $N > \rho \sqrt{3}$. Here $e_1, e_2, e_3$ are the standard basis vectors in $\real^3$.
 		
 		Our next result addresses the recovery of $a,b$ when $q$ is known.
 		\begin{theorem}[Stability for $a,b$]\label{thm:abstab}
 		Suppose $q$ is a compactly supported smooth function in $\mb{R}^3 \times \real$ with support in $\Bbar \times [0,T]$, and $\xi_1, \cdots, \xi_4$ is a diverse set of locations with respect to $B$.
 		Given $M>0$, if $a, \aacu, b, \bacu$ are smooth functions and vector fields on $\real^3 \times \real$ with support in $\Bbar \times [0,T]$ and
 		 $\|[a, b, \acute a, \acute b, q]\|_{C^{19}(\overline{B}\times [0,T])}\le M$, we have
 		\begin{align*}
 		 \|[a - \acute a, b - \acute b]\|_{0, \real^3 \times [0,T] }  
 		~ \cleq ~
 		 \sum_{i = 1}^{4} \int_{ -1 - |\xi_i|}^{T + 1 - |\xi_i| }  & \left(\| (u - \acute u) (\cdot, T; \xi_i, \tau) \|_{1,H_{\xi_i, \tau}} \right. \\
 		 & \hspace{2cm} \left. + \|(u_t - \acute u_t) (\cdot, T; \xi_i, \tau)\|_{0,H_{\xi_i, \tau}} \right) \d \tau.
 		\end{align*}
 		Here the constant is independent of $a,b,\aacu, \bacu$, and $u$, $\acute u$ are the functions associated to $(a,b,q)$ and $(\aacu, \bacu, q)$ guaranteed by Proposition \ref{Heaviside}.
 		\end{theorem}

 		Our next result addresses the uniqueness in the recovery of $(a, b, c)$. As shown earlier, 
 		one expects to recover only $\text{curl}(a,b)$ and $c$. Unfortunately, to obtain this result we need to restrict $a,b$ to those for which $a + \theta_4 \cdot b$ and $a_t + \theta_4 \cdot b_t$ satisfy a certain integral relation; here
 		\[
 		\theta_4(x) = \frac{x-\xi_4}{|x-\xi_4|}, \qquad x \in \Bbar.
 		\]
 		This relation and the proof of the uniqueness result were inspired by a relation and an argument in \cite{MPS2020}, where a similar uniqueness question was studied though in the time-independent setting.   
 		%
 		
 		
 		There is a gauge invariance associated with the problem of recovering $a,b,c$.
 		If $\phi(x,t)$ and $f(x,t)$ are smooth functions on $\mb{R}^3 \times \mb{R}$, we have
 		\begin{align*}
 		    \left(\dd_t - a - \phi_t\right)(e^{\phi} f) = e^{\phi} \left(\dd_t - a\right)f, \qd \left(\nabla - b - \nabla \phi\right)(e^{\phi} f) = e^{\phi} \left(\nabla - b \right)f,
 		\end{align*}
 		resulting in 
 		\begin{align*}
 		    \mc{L}_{a + \phi_t, b + \nabla \phi, c} (e^\phi f) = e^\phi \mc{L}_{a, b, c}f.
 		\end{align*}
 		Hence, if $\phi(\xi, t) = 0$ for $t \in \mb{R}$, we have
 		\begin{align*}
 		    \mc{L}_{a + \phi_t, b + \nabla \phi, c} (e^\phi U) = e^\phi \mc{L}_{a, b, c}(U) = 4 \pi e^\phi H(t - \tau) \delta(x - \xi) = 4 \pi H(t - \tau) \delta(x - \xi)
 		\end{align*}
 		and
 		\begin{align*}
 		    \mc{L}_{a + \phi_t, b + \nabla \phi, c} (e^\phi V) = 4 \pi \delta(t - \tau, x - \xi).
 		\end{align*}
 		As a consequence, $\Fcal(a, b, c) = \Fcal(a + \phi_t, b + \nabla \phi, c)$ for any 
 		smooth function $\phi(x,t)$ with support in $\Bbar \times [0,T]$ and $\phi(\cdot,T)=0, \phi_t(\cdot,T)=0$. This suggests we can hope to recover at most the curl of $[a,b]$,
 		that is $d ( a dt + b^1 dx^1 + b^2 dx^2 + b^3 dx^3)$.
 		
 		\begin{theorem}[Uniqueness for $\tn{curl} (a,b)$ and $c$]\label{thm:abcunique}
 		Suppose $a,c, \aacu, \cacu$ and
 		$b,\bacu$ are smooth functions and vector fields
 		on $\real^3 \times \real $ with support in $\Bbar \times [0,T]$.
 		Let $\xi_1, \cdots, \xi_4$ be a diverse set of locations with respect to $(T+1)B$
 		and $u, \acute u$ and 
 		$v, \acute v$ the functions associated with $(a,b,c)$ and $(\aacu, \bacu, \cacu)$, respectively, guaranteed by
 		Propositions \ref{Heaviside} and \ref{Delta}. If 
 		\begin{align*}
 		    & [u - \uacu, (u - \uacu)_t](x, T, \xi_i, \tau) = 0, \qd \forall x\in H_{\xi_i, \tau},\ \tau\in [-1 - |\xi_i|, T + 1 - |\xi_i|], \ i\in\{1,2,3,4\},\\
 		    & [v - \vacu, (v - \vacu)_t](x, T, \xi_4, \tau) = 0,
 		    \qd \forall  x\in H_{\xi_4, \tau},\ \tau\in [- 1 - |\xi_4|, T + 1 - |\xi_4|],
 		\end{align*}
 		and 
 		\begin{align*}
 		   & \int_{0}^{|x-\xi_4|} \left( (a-\aacu) + \theta_4 \cdot (b-\bacu) \right) (x - s \theta_4, T - s)\  \d s = 0, \qd \forall x\in \mb{R}^3,\\
 		   & \int_{0}^{|x-\xi_4|} \left( (a-\aacu)_t + \theta_4 \cdot (b - \bacu)_t \right) 
 		   (x - s \theta_4, T - s)\  \d s = 0, \qd \forall x\in \mb{R}^3,
 		\end{align*}
 		then
 		\begin{align*}
 		   \d \left( a\d t+ \sum_{i=1}^{3}b^i\d x^i\right) = \d \left(\acute a\d t+ \sum_{i=1}^{3}\acute b^i\d x^i\right), \qd c = \acute c. 
 		\end{align*}
 		\end{theorem}
 		Note that we use data from the $u, \uacu$ solutions for all four source locations $\xi_1, \cdots, \xi_4$
 		but we use data from the $v, \vacu$ solutions only for the source at $\xi_4$.

 		We also have a Lipschitz stability result for the recovery of  $\tn{curl} (a,b)$ and $c$. However, we require more data than was needed for the uniqueness result in Theorem \ref{thm:abcunique}. Let $\psi$ be the solution of the IVP
 		\begin{align}
 		   \square \psi = c - a_t + \nabla \cdot b, & \qd \tn{ in } \real^3 \times (-\iy, T];
 		   \label{extra data}
 		   \\
 		   \psi(\cdot,t)=0, & \qd t <0. \label{eq:psiic}
 		\end{align}
 		For the stability result, in addition to $\Fcal(a,b,c)$, we need the traces of
 		$\psi, \psi_t, \psi_{tt}$ on $t=T$; this replaces the integral condition used in Theorem
 		\ref{thm:abcunique}. We do not know whether there is stability without this extra data.
 		\begin{theorem} [Stability for $\tn{curl} (a,b)$ and $c$]\label{thm:abcstab}
 		Suppose $\xi_1, \cdots, \xi_4$ is a diverse set of locations with respect to $(T+1)B$.
 		Given $M>0$, if $a,c, \acute a,\acute c$ and $b, \bacu$ are smooth functions and vector fields
 		 on $\mb{R}^3 \times \real$ with support in $\overline{B} \times [0,T]$
 		 and $\| [a, b, c, \acute a, \acute b, \acute c] \|_{C^{22}(\overline{B}\times [0,T])}\le M$, then
 		\begin{align*}
 		    \|[\d \eta - \d \acute \eta, & c - \acute c]\|_{L^2(\real^3 \times [0,T]) } 
 		    \\
 		    & \cleq \sum_{i = 1}^{4} \int_{ -1 - |\xi_i|}^{T+1-|\xi_i| } \left( \| (u - \acute u) (\cdot, T; \xi_i, \tau) \|_{2,H_{\xi_i, \tau}} + \|(u_t - \acute u_t) (\cdot, T; \xi_i, \tau)\|_{1,H_{\xi_i, \tau}} \right) \d \tau \\ 
 		    & + \sum_{i = 1}^{4} \int_{ -1 - |\xi_i|}^{T+1-|\xi_i| } \|\left( u_{tt} - \acute u_{tt} \right) (\cdot, T; \xi_i, \tau )\|_{0, H_{\xi_i,\tau}}\ \d \tau \\
 		    & + \sum_{i = 1}^{4} \int_{ -1 - |\xi_i|}^{T+1-|\xi_i| } \left( \| (v- \acute v) (\cdot, T; \xi_i, \tau) \|_{1,H_{\xi_i, \tau}} + \|(v_t - \acute v_t) (\cdot, T; \xi_i, \tau)\|_{0,H_{\xi_i, \tau}} \right) \d \tau \\
 		    & + \|(\psi - \acute \psi) (\cdot, T)\|_{2, \mb{R}^3} + \|(\psi_t - \acute \psi_t)(\cdot, T)\|_{1, \mb{R}^3} + \|(\psi_{tt} - \acute \psi_{tt})(\cdot, T)\|_{0, \mb{R}^3},
 		\end{align*}
 	where $\eta$ and $\acute \eta$ are the 1-forms
 		\begin{align*}
 		    \eta = a\d t+ \sum_{i=1}^{3}b^i\d x^i, \qd \acute \eta = \acute a\d t+ \sum_{i=1}^{3}\acute b^i\d x^i
 		\end{align*}
 		and the constant is independent of $a,b,c, \aacu, \bacu, \cacu$. Here $\psi, \psi'$ are the solutions of the IVP (\ref{extra data}), (\ref{eq:psiic})
 		and $u, \acute u, v, \acute v$ are the functions corresponding to $(a,b,c)$ and $(\acute a, \acute b, \acute c)$
 		guaranteed by Propositions \ref{Heaviside} and \ref{Delta} . 
 		\end{theorem}

 		 A fundamental aspect of our work is the Lipschitz stability results for the space and time dependent coefficients obtained by the use of Carleman estimates on domains depending on the parameter $\tau$ and the integral of these estimates with respect to $\tau$. We very much exploit the relation between the unknown coefficients and the traces of the solutions of the IVP on the characteristic cones.
 		
 		 We introduce some notation used in the rest of the article.
 		 For convenience, we  denote the operators 
 		 $\mc{L}_{a,b,c}$ and $\mc{L}_{\acute a, \acute b, \acute c}$ by $\mc{L}$ and $\acute {\mc{L}}$ respectively. We define the differences 
 		\begin{align}\label{differences}
 		    \overline a := a - \acute a, \qd \overline b := b - \acute b, \qd \overline c := c - \acute c, \qd \overline q := q - \acute q, \qd \overline u := u - \acute u, \qd \overline v := v - \acute v. 
 		\end{align}
 	Also, given a $\xi \in \real^3 \setminus \Bbar$ and $x \in \Bbar$, recall that we have defined
 	\[
 	\theta(x) := \frac{x-\xi}{|x-\xi|}, \qquad x \in \real^3, ~ x \neq \xi.
 	\]
 	We use $\theta$ instead of $\theta(x)$ most of the time and we use $\theta_i$ when $\xi$ is replaced by 
 	$\xi_i$.
 	
 	A key ingredient of the proofs of the theorems is a Carleman estimate, with explicit boundary terms, for the operator $\Lcal_{a,b,c}$, in the region $Q_{\xi,\tau}$. We state it here and give the proof in Section \ref{sec:carleman}.
 		\begin{proposition}\label{prop:carleman}
 		Suppose $\tau \in \mb{R}$, $\xi \in \real^3$ and $a,q$, $b$ are smooth functions and 
 		vector fields in $\real^3 \times \real$ with $\{\xi \} \times \real$ disjoint from the supports of $a,b,q$.
 	 Then there is a $\sigma_0 >0$ so that
 		\begin{align}
 		 \nn   \sigma \int_{Q_{\xi, \tau}} e^{ 2\sigma t} & \left( |\nabla_{x,t} w|^2 + \sigma^2 w^2  \right)  + \sigma \int_{C_{ \xi, \tau}} e^{ 2\sigma t} \left( |\nabla_{C} w|^2 + \sigma^2 w^2  \right) \\
 		 \label{carleman}   & \qquad \cleq \left( \int_{Q_{\xi, \tau}} e^{ 2\sigma t} |\mc{L}_{a,b,q} w|^2  + \sigma \int_{H_{\xi,\tau}} e^{ 2\sigma t} \left( |\nabla_{x,t} w|^2 + \sigma^2 w^2  \right) \right),
 		\end{align}
 		for every $w \in C^3(Q_{\xi, \tau})$ and every $\sigma \geq \sigma_0$.
 		Here $\nabla_{C}$ represents the gradient on the submanifold $C_{\xi, \tau}$. Further, 
 		the constant is independent of $w$ and $\sigma$ and depends only on $T,|\xi|, |\tau|$ and 
 		$\|[a,b,q]\|_{C^0(Q_{\xi,\tau})}$.
 		\end{proposition}
 		
The rest of the article gives the proofs of the propositions and the theorems stated above.
 		
 		
\section{Proof of Theorem \ref{thm:qstab}} \label{sec:qstab}

 		We have $a = \acute a,\ b = \acute b$ and a single source
 		$\xi \in \mb{R}^3 \setminus \overline B$. For any $\tau \in [-(1+|\xi|), T+1 - |\xi|]$, 
 		let $v, \vacu$ be
 		the functions guaranteed by Proposition \ref{Delta} for the coefficients $(a, b, q)$ and $(a, b, \acute q)$. Taking the
 		differences of (\ref{ch BVP for delta 1}), (\ref{ch BVP for delta 2}) for the two sets of coefficients, we obtain
 		\begin{align}
 		 \label{diff eqn for q stability 1}  \mc{L} \overline v = - \overline q \acute v, 
 		 & \qd \tn{ in } Q_{\xi,\tau}, \\
 		 \label{diff eqn for q stability 2}   
 		 2 \left(  \dd_t + \theta \cdot \nabla - (a+\theta \cdot b)
 		 + r^{-1} \right) \overline v = - \overline q \alpha,& \qd \tn{ on } C_{\xi,\tau},
 		\end{align}
 		{ where $\overline{q},\overline{v}$ were defined in \eqref{differences}}. Applying Proposition \ref{prop:carleman} to the function $\overline v$ in the region 
 		$Q_{\xi,\tau}$, we have
 		\begin{align}\nn
 		    \sigma \|\overline v\|^2_{1, \sigma, C_{\xi,\tau}} 
 		    \cleq \|\mc{L}\overline v\|^2_{0, \sigma, Q_{\xi,\tau}} + \sigma \|\overline v\|^2_{1, \sigma, H_{\xi, \tau}} + \sigma \|\overline v_t\|^2_{0, \sigma, H_{\xi, \tau}},
 		\end{align}
 		with the constant dependent only on $|\xi|, T$, $\|[a,b,q]\|_{C^0(Q_{\xi,\tau})}$.
 		Hence, using \eqref{diff eqn for q stability 1},
 		\eqref{diff eqn for q stability 2} and that $r\alpha$ is a positive continuous function on $\overline Q_{\xi, \tau}$, we have 
 		\begin{align}\label{pre q stability}
 		\sigma \|\overline q\|^2_{0, \sigma, C_{\xi, \tau}} 
 		& ~ \cleq ~ \sigma \| r \alpha \qbar \|^2_{0, \sigma, C_{\xi, \tau}}
 		~ \cleq ~ \sigma 
 		\|r (\partial_t + \theta \cdot \nabla -(a+ \theta \cdot b) + r^{-1} )\vbar \|^2_{0,\sigma,C_{\xi, \tau}}
 		\nonumber
 		\\
 		& ~\cleq ~
 		\sigma \|\overline v\|^2_{1, \sigma, C_{\xi,\tau}} 
 		~\cleq~
 		\|\overline q\|^2_{0, \sigma, Q_{\xi, \tau}} 
 		+ \sigma \|\overline v\|^2_{1, \sigma, H_{\xi, \tau}} 
 		+ \sigma \|\overline v_t\|^2_{0, \sigma, H_{\xi, \tau}}.
 		\end{align}
Here the constant depends on $|\xi|, T$, $\|[a,b,q]\|_{C^0(Q_{\xi,\tau})}$ and
$\|\vacu\|_{C^0}$. Using Proposition \ref{Delta}, the constant depends only on $T, M$ and $\vert \xi \vert$.

 		Noting that $q, \qacu$ are supported in $\Bbar \times [0,T]$ and $\Bbar \times [0,T]$ does not intersect $C_{\xi,\tau}$ if $\tau$ is outside the interval $[-(1+|\xi|), T+1 - |\xi|]$, we 
 		have
 		\begin{align*} 
 		 \int_{-1 - |\xi| }^{T + 1 - |\xi|}  \int_{C_{\xi, \tau}} e^{2 \sigma t} \,
 		 |\overline q(x,t)|^2 \ & \d S\ \d\tau 
 		  = \int_{\real}  \int_{C_{\xi, \tau}} e^{2 \sigma t} \,
 		 |\overline q(x,t)|^2 \ \d S\ \d\tau 
 		 \\
         & = \sqrt{2} \int_\real \int_{\mb{R}^3 \times \mb{R}} e^{2 \sigma t} \,
         |\overline q(x,t)|^2 \, \delta( t - \tau - |x - \xi|) \ \d x \d t\ \d \tau 
 		 \\
 		 & = \sqrt{2}\int_{\mb{R}^3 \times \mb{R}} e^{2 \sigma t} \, |\overline q(x,t)|^2 \int_{\mb{R}} \delta( t - \tau - |x - \xi| )\ \d \tau  \ \d x \d t 
 		 \\
 		 & = \sqrt{2}\int_{\mb{R}^3 \times \mb{R}} e^{2 \sigma t} \, |\overline q(x,t)|^2\ \d x \d t
 		 \\
 		 & = \sqrt{2}\, \|\overline q\|^2_{0, \sigma, \real^3 \times [0,T]}.
 		\end{align*}
 		Hence integrating \eqref{pre q stability} w.r.t $\tau$ over the interval $\left[-1 - |\xi|, T + 1 - |\xi|\right]$, we obtain
 		\begin{align*}
 		    \sigma \|\overline q\|^2_{0, \sigma, \overline B \times [0,T]} ~ \cleq ~ \|\overline q\|^2_{0, \sigma, \overline B \times [0,T]} + \sigma \int_{- 1 - |\xi|}^{T + 1 - |\xi|} \left(\|\overline v\|^2_{1, \sigma, H_{\xi, \tau}} + \|\overline v_t\|^2_{0, \sigma, H_{\xi, \tau}}\right)\ 
 		    \d \tau,
 		\end{align*}
 		which proves Theorem \ref{thm:qstab} if we take $\sigma $ to be  large enough. 
 		
 		
 \section{Proof of Theorem \ref{thm:abstab}}\label{sec:abstab}
 		Fix a $\xi_i$ and a $\tau \in [-(1+|\xi_i|), T+1 - |\xi_i|]$. Let $u, \uacu$ be the solutions corresponding to the coefficients $(a,b,q)$ and 
 		$(\acute a, \acute b, q)$ guaranteed by Proposition \ref{Heaviside}. Note that $u, \uacu$ are zero
 		in a neighborhood of $(x \smeq \xi_i, t \smeq \tau)$ and $a,b,q, \aacu, \bacu, \qacu$ are supported
 		away from $\{ \xi \} \times \real$.
 		
 		Taking the differences of versions of 
 		\eqref{ch BVP for Heaviside 1}, \eqref{ch BVP for Heaviside 2} 
 		associated to $(a, b, q)$ and $(\aacu, \bacu, q)$, for each $\xi_i$, we have 
 		\begin{align}
 		 \label{eqn 1 for (a,b) stability}   & \mc{L} \overline u = 2 \overline a \acute u_t - 2 \overline b \cdot \nabla \acute u + \f{2 \overline b\cdot (x - \xi_i)}{|x - \xi_i|^3},  \qd \tn{ in } Q_{\xi_i, \tau},  \\
 		& \overline u = \alpha - \acute \alpha, \qd \tn{ on } C_{\xi_i, \tau}.
 		\label{eq:ubardiff}
 		\end{align}
 		
 		Applying Proposition \ref{prop:carleman} to $\ubar$ in the region $Q_{\xi_i, \tau}$, we have
 		\begin{align}
 		\sigma \|\overline u\|^2_{1, \sigma, C_{\xi_i, \tau}} \cleq 
 		 \|\mc{L} \ubar \|^2_{0, \sigma, Q_{\xi_i, \tau}} + \sigma \|\overline u\|^2_{1, \sigma, H_{\xi_i,\tau}} + \sigma \|\overline u_t\|^2_{0, \sigma, H_{\xi_i,\tau}},
 		 \label{eq:ubarcarlabq}
 		\end{align}
 		with the constant dependent only on $\xi_i, T, \|[a,b,q]\|_{C^0}$.
 		
 		Now, using (\ref{eq:Talpha}), we have
 		\begin{align}
 		 ( \dd_t + \theta \cdot \nabla & - (a + \theta_i \cdot b) + r^{-1}) (\alpha -\acute \alpha)
 		 \nn
 		 \\
 		 & = - ( \dd_t + \theta \cdot \nabla  - (a + \theta_i \cdot b) + r^{-1}) \alphaacu
 		 \nn
 		 \\
 		 & = - ( \dd_t + \theta \cdot \nabla  - (\aacu + \theta_i \cdot \bacu) + r^{-1}) \alphaacu
 		 - (\abar + \theta_i \cdot \bbar) \alphaacu
 		 \nn
 		 \\
 		 & = - (\abar + \theta_i \cdot \bbar) \alphaacu.
 		  \label{eqn 2 for (a,b) stability}
 		\end{align}
 		So, using \eqref{eqn 1 for (a,b) stability} - \eqref{eqn 2 for (a,b) stability} and noting that $\bbar$ is supported away 
 		from $\{\xi_i\} \times \real$, we obtain
 		\begin{align}\label{pre-stability of (a,b) 1}
 		   \sigma \|\overline a + \theta_i \cdot \overline b\|^2_{0, \sigma, C_{\xi_i, \tau}} 
 		   \cleq \sigma \|\overline u\|^2_{1, \sigma, C_{\xi_i, \tau}} \cleq \|[\overline a, \overline b]\|^2_{0, \sigma, Q_{\xi_i, \tau}} + \sigma \|\overline u\|^2_{1, \sigma, H_{\xi_i,\tau}} + \sigma \|\overline u_t\|^2_{0, \sigma, H_{\xi_i,\tau}},
 		\end{align}
 		with the constant dependent only on $\xi_i, T, \|[a,b,q]\|_{C^0}$ and $\|\uacu\|_{C^1}$, hence on $\vert \xi_4\vert, T$ and $M$.
 		
 		Imitating the integral relation calculation in the proof of Theorem \ref{thm:qstab}, we have
 		\[
 		\int_{-1-|\xi_i|}^{T+1-|\xi_i|} \int_{C_{\xi_i, \tau}} e^{2 \sigma t} \,
 		|\abar + \theta_i \cdot \bbar|^2 \, dS_{x,t} \, d \tau
 		= \sqrt{2} \int_{\Bbar \times [0,T]} e^{2 \sigma t} \, |\abar + \theta_i \cdot \bbar|^2 \, dx \, dt.
 		\]
 		Hence integrating \eqref{pre-stability of (a,b) 1} w.r.t $\tau$ over 
 		$[-1 - |\xi_i|, T + 1 - |\xi_i|]$
 		\begin{align}\label{pre-stability of (a,b)}
 		    \sigma \|\overline a + \theta_i \cdot \overline b\|^2_{0, \sigma, \overline B \times [0,T]} 
 		    \cleq \|[\overline a, \overline b]\|^2_{0, \sigma, \overline B \times [0, T]} + \sigma \int_{- 1 - |\xi_i|}^{T + 1 - |\xi_i|} \left( \|\overline u\|^2_{1, \sigma, H_{\xi_i,\tau}} + \|\overline u_t\|^2_{0, \sigma, H_{\xi_i,\tau}} \right)\ \d\tau.
 		\end{align}
 		Using \eqref{pre-stability of (a,b)} for each $i=1,2,3,4$ and noting that
 	 $\xi_i$, $i=1, \cdots, 4$, is a diverse set of locations w.r.t $B$, we obtain
 		\begin{align*}
 		    \sigma \|[\overline a, \overline b]\|^2_{0, \sigma, \overline B \times [0, T]} 
 		    \cleq \|[\overline a, \overline b]\|^2_{0, \sigma, \overline B \times [0, T]} + \sigma \sum\limits_{i = 1}^4\int_{ - 1 - |\xi_i|}^{T + 1 - |\xi_i|} \left( \|\overline u\|^2_{1, \sigma, H_{\xi_i,\tau}} + \|\overline u_t\|^2_{0, \sigma, H_{\xi_i,\tau}} \right)\ \d\tau
 		\end{align*}
 		The theorem follows if we choose $\sigma$ large enough.

 		
\section{Proof of Theorem \ref{thm:abcunique}}
 	    We seek to prove the uniqueness in the recovery $\tn{curl}(a,b)$ and $c$ from the values of
 	    $U,U_t, V, V_t$ on $H_{\xi, \tau}$ for $\xi = \xi_i$, $i=1, \cdots, 4$. 
	    
	    {
	    At first glance it would seem that we just need to imitate and combine the arguments used in the proofs of Theorems \ref{thm:qstab} and \ref{thm:abstab}. In fact, imitating the approach in the proof of Theorem \ref{thm:abstab}, one obtains the estimate 
	    \beqn
	     \sigma \| [\abar, \bbar] \|_{0,\sigma} \cleq \|[\abar, \bbar, \qbar]\|_{0, \sigma}.
	     \label{eq:expl16a}
	     \eeqn
	     However, if we imitate the proof of 
	    Theorem \ref{thm:qstab}, there is a difficulty in obtaining the estimate
	    \[
	    \sigma \|\qbar\|_{0,\sigma} \cleq  \|\qbar\|_{0,\sigma} + \sigma \|[\abar, \bbar]\|_{0,\sigma}.
	    \] 
	     Before, we describe the difficulty and how we resolve it, we note that the two estimates mentioned above would 
	     give us stability in the inverse problem of determining $a, b, c$, which of course is not feasible since the problem is gauge 
	     invariant and one can hope to recover only curl$(\abar, \bbar)$ and $\cbar$.
	     
	     Now the difficulty mentioned in the previous paragraph is addressed. To estimate $\qbar$ we would use the PDE for $\vbar$, 	     
	     the difference of \eqref{ch BVP for delta 2} for $a,b,q$ and $\aacu, \bacu, \qacu$, combined with a hoped for estimate
	     of the form
	     \beqn
	     |q - \qacu| \cleq | \mc{L}_{a,b,q} \alpha - \mc{L}_{\aacu, \bacu, \qacu} \acute{\alpha}| + |[\abar, \bbar]|
	     \label{eq:tempaabbqq}
	     \eeqn
	     This was straightforward in the proof of Theorem \ref{thm:qstab} because there $a=\aacu, b = \bacu$ and hence
	     $\alpha = \acute{\alpha}$ (which was bounded away from $0$); therefore the RHS of \eqref{eq:tempaabbqq}
	     was just $|q-\qacu| \alpha$. For Theorem \ref{thm:abcunique}, we do not have $a=\aacu, b = \bacu$ or $\alpha = \alphaacu$. 
	     In this case, instead of \eqref{eq:tempaabbqq} one can show
	     \beqn
	     |q - \qacu| \cleq | \mc{L}_{a,b,q} \alpha - \mc{L}_{\aacu, \bacu, \qacu} \acute{\alpha}| + |[\abar, \bbar, 
	     \nabla_{x,t} (\abar + \theta \cdot \bbar)]|,
	     \label{eq:t2aabb}
	     \eeqn
	    because $\mc{L}_{a,b,q} \alpha$ requires first order derivatives of $a+ \theta \cdot b$. An estimate of the type \eqref{eq:t2aabb} is 
	    not useful since our other estimate \eqref{eq:expl16a}, is for the zeroth order norm of $\abar, \bbar$.
	    
	    The article  \cite{MPS2020} showed the way to resolve this difficulty; replace $\abar, \bbar$ by a
	    gauge equivalent $\abar, \bbar$  for which $\abar+ \theta \cdot \bbar=0$ (for one of the $\theta$) and the issue of first order 
	    derivatives of $\abar + \theta \cdot \bbar$ in \eqref{eq:t2aabb} goes away. Such gauge equivalent $\abar, \bbar$ exist but the gauge 
	    affects the data unless we impose conditions on the gauge - that is the null integral condition the hypothesis of Theorem 
	    \ref{thm:abcunique}.
	    
	    We now proceed with the proof of the theorem.
	    }
 	    
 		The $u, \uacu, v, \vacu$ are the solutions guaranteed by Propositions \ref{Heaviside} and
 	    \ref{Delta} for the coefficients $a,b,c$ and  $\aacu, \bacu, \cacu$.
 	    Since $a,b,c$ are supported in $\Bbar \times [0,T]$ and 
 		$\xi \in \real^3 \setminus (T+1)\Bbar$, one may check that for a fixed $\xi \in (T+1)B$, 
 		the values of $u, u_t, v, v_t$ 
 		and $\uacu, \uacu_t, \vacu, \vacu_t$ on $\real^3 \times \{t \smeq T\} $ 
 		are zero for $\tau>T+1 - |\xi|$
 		and do not change as $\tau$ varies over $\tau \in (-\infty, -1 - |\xi|]$. Hence,
 		from the hypothesis of Theorem \ref{thm:abcunique}, we may assume that
 		\begin{align*}
 		    & [u - \uacu, (u - \uacu)_t](x, T, \xi_i, \tau) = 0, \qd \forall x\in H_{\xi_i, \tau},\ i\in\{1,2,3,4\},\\
 		    & [v - \vacu, (v - \vacu)_t](x, T, \xi_4, \tau) = 0,
 		    \qd \forall  x\in H_{\xi_4, \tau},
 		\end{align*}
 		for all $\tau \in \real$ rather than for a limited range of $\tau$.

 	    As discussed in the introduction, due to gauge invariance,
 	    there is a  natural obstruction to uniqueness when attempting to
 	    recover $a,b,c$. For $x \neq \xi_4$, define
 	    \begin{align*}
 	        \phi(x,t) & := - \int_{0}^{|x - \xi_4|} (a + \theta_4 \cdot b) (x - s\theta_4, t - s) \ \d s,
 	        \\
 	       \acute \phi(x,t) &= - \int_{0}^{|x - \xi_4|} (\acute a + \theta_4 \cdot \acute b) (x - s\theta_4, t - s) \ \d s.
 	    \end{align*}
 	    Since $\xi_4 \notin (T+1)\Bbar$ and $a,b,c, \aacu, \bacu, \cacu$ are supported in
 	    $\Bbar \times [0,T]$, we have $\phi=0$ and $\phiacu=0$
 	    in a punctured cylindrical neighborhood of $\{ \xi_4\} \times \real$. Hence, 
 	    defining $\phi(\xi_4,\cdot):=0$, $\phiacu(\xi_4, \cdot):=0$ gives us smooth functions on
 	    $\real^3 \times \real$ which are zero in a cylindrical neighborhood of $\{0\} \times \real$. We also note that the intersection of the supports of $\phi, \phiacu$ with $\real^3 \times (-\infty,T]$ is contained in $(T+1) \Bbar \times (-\infty, T]$.
 	    
 	    As shown in the introduction, $e^\phi u$, $e^\phiacu \uacu$
 	    $e^\phi v$, $e^\phiacu \vacu$ are the functions guaranteed by
 	    Propositions \ref{Heaviside} and \ref{Delta} for the coefficients
 	    $a+\phi_t, b + \nabla \phi, c$ and $\aacu + \phiacu_t, \bacu + \nabla \phiacu, \cacu$, provided
 	    the hypotheses of these propositions are satisfied. The propositions require that the cylinder
 	    $\{\xi\} \times \real$ not intersect the supports of the coefficients, which seems not to be true
 	    for the modified $a,b,c$. However, we will be using the values of $e^\phi u$, $e^\phiacu \uacu$
 	    $e^\phi v$, $e^\phiacu \vacu$ only on subsets of the region 
 	    $\real^3 \times (-\infty,T]$, so we need Propositions \ref{Heaviside} and \ref{Delta} only for the region $\real^3 \times (-\infty,T]$. Since $|\xi_i|>T+1$, the cylinders 
 	    $\{ \xi_i \} \times (-\infty, T]$, $i=1,2,3,4$, do not intersect the supports of 
 	    $a+\phi_t, b + \nabla \phi, c$ and $\aacu + \phiacu_t, \bacu + \nabla \phiacu, \cacu$, so
 	    Propositions \ref{Heaviside} and \ref{Delta} are valid on $\real^3 \times (-\infty, T]$
 	    for the coefficients $a+\phi_t, b + \nabla \phi, c$ and $\aacu + \phiacu_t, \bacu + \nabla \phiacu, \cacu$.
 	    
 From our hypothesis, we have
 		\[
 		\phi(x,T)=\phiacu(x,T), ~~ \phi_t(x,T)=\phiacu_t(x,T) 
 		\qquad
 		\]
 Hence, on $ H_{\xi_i, \tau}$, $i=1,2,3,4$, we have
 		\begin{align*}
 		    & [e^{\phi} u, (e^{\phi} u)_t](\cdot, T, \xi_i, \tau)] = [e^{\acute \phi} \acute u, (e^{\acute \phi} \acute u)_t](\cdot, T, \xi_i, \tau)], \qd 
 		   \forall \tau\in \real,\\
 		    & [e^{\phi} v, (e^{\phi} v)_t](\cdot, T, \xi_4, \tau)] = [e^{\acute \phi} \acute v, (e^{\acute \phi} \acute v)_t](\cdot, T, \xi_4, \tau)],\qd \forall \tau\in \real.
 		\end{align*}
 		 Thus to prove Theorem \ref{thm:abcunique}, it suffices to work with the modified
 		 coefficients $(a + \phi_t, b + \nabla \phi, c)$ and
 		 $(\aacu + \phiacu_t, \bacu + \nabla \phiacu, \cacu)$ 
 		 because $(a + \phi_t, b + \nabla \phi)$ has {the} same curl as $(a, b)$
 		 and $(\aacu + \phiacu_t, \bacu + \nabla \phiacu)$ has the same
 		 curl as $(\aacu, \bacu)$.
 		 Further, since $a,b, \aacu, \bacu$ are supported away from $\{\xi\} \times \real$, we have
 		 \begin{align*}
 		  \left(\dd_t + \theta_4 \cdot \nabla \right) \phi & =  \int_{0}^{|x - \xi_4|} \f{\d}{\d s} ( a + \theta_4 \cdot b) (x - s\theta_4, t - s) \ \d s\\
 		  & = - ( a + \theta_4 \cdot b)(x, t)
 		\end{align*}
 		which implies
 		\begin{align*}
 		  (a + \phi_t) + \theta_4 \cdot (b + \nabla \phi) = 0,
 		  \qquad 
 		  (\aacu + \phiacu_t) + \theta_4 \cdot (\bacu + \nabla \phiacu) 
 		  = 0
 		\end{align*}
 		So, to prove Theorem \ref{thm:abcunique}, we may assume that
 		\begin{align}
 		    a + \theta_4 \cdot b = \acute a + \theta_4 \cdot \acute b = 0.
 		    \label{eq:atbzero}
 		\end{align}
 		We show that $(a,b,q) = (\aacu, \bacu, \qacu)$ for these modified triples, which will prove the theorem. There is a subtle point here which we discuss next.
 	    
 	    { We have replaced $a,b$ (and do the same for $\aacu, \bacu$) by a special $a,b$ (for which $a+\theta_4 \cdot b=0$), 
	    which, just for this paragraph, we call $\tilde{a}:=a+\phi_t, \tilde{b} := b+\nabla \phi$. Now 
	    $\text{curl}(\tilde{a}, \tilde{b}) = \text{curl}(a,b)$  
	    and the corresponding values of $\tilde{u},\tilde{u}_t, \tilde{v}, \tilde{v}_t$ and 
	    $\tilde{\uacu}, \tilde{\uacu}_t, \tilde{\vacu}, \tilde{\vacu}_t$ agree on $H_{\xi, \tau}$ because they agree for 
	    $u, u_t, v, v_t$ and $\uacu, \uacu_t, \vacu, \vacu_t$ and also for $\phi, \phi_t$ and $\phiacu, \phiacu_t$.
	    We have not modified $c, \cacu$. Now $\tilde{q} = c -\tilde{a}_t + \nabla \cdot \tilde{b} 
	    + |\tilde{a}|^2 - |\tilde{b}|^2$ and there is a similar expression for $\tilde{\qacu}$. If we show that
	    $\tilde{a} = \tilde{\aacu}, \tilde{b} = \tilde{\bacu}, \tilde{q} = \tilde{\qacu}$, then we get $c=\cacu$ and
	    $\text{curl}(a,b) = \text{curl}(\tilde{a}, \tilde{b}) = \text{curl}(\tilde{\aacu}, \tilde{\bacu}) = \text{curl}(\aacu,\bacu)$, proving the theorem, 
	    but note that we cannot claim $q=\qacu$.}
 		 
 		Define
 		\[
 		\tau_{min} := \text{min} \{ - (|\xi_i| + T+1): i=1,2,3,4 \},
 		\qquad
 		\tau_{max} := \text{max} \{ 2T+1 - |\xi_i|: i=1,2,3,4\},
 		\]
 		and our $\tau$ will vary in the interval $[\tau_{min}, \tau_{max}]$. For use below, 
 		observe that
 		the intersection of the supports of the modified $a,b,q$ with $\real^3 \times (-\infty,T]$ 
 		is contained in $(T+1)\Bbar \times [0,T]$ and the union of the $C_{\xi,\tau}$, as $\tau$ varies
 		over $[\tau_{min}, \tau_{max}]$, contains $(T+1)\Bbar \times [0,T]$.
 		
 		From \eqref{ch BVP for Heaviside 1} - \eqref{ch BVP for Heaviside 2},
 		for each $i=1, \cdots, 4$, we obtain
 		\begin{align}
 		   \mc{L} \overline u = 2 \overline a \acute u_t - 2 \overline b \cdot \nabla \acute u - \overline q \acute u + \f{2 \overline b \cdot (x - \xi_i)}{|x - \xi_i|^3} - \f{\overline q (x,t)}{|x - \xi_i|}, 
 		   & \qd \tn{ in } Q_{\xi_i,\tau},
 		   \nn
 		   \\
 		   \overline u = \alpha - \acute \alpha, 
 		   &\qd \tn{ on } C_{\xi_i, \tau}.
 		   \label{eq:ubabqunique}
 		\end{align}
 		Noting that $\{\xi_i\} \times(-\infty, T]$ does not intersect the supports of $\bbar$ and $\qbar$,
 		we have 
 		\[
 		\|\mc{L} \ubar\|_{0, \sigma, Q_{\xi, \tau}} \cleq \,
 		\|[\abar, \bbar, \qbar]\|_{0, \sigma, Q_{\xi_i,\tau}}.
 		\]
 		So applying Proposition \ref{prop:carleman} to $\ubar$ on the region $Q_{\xi_i,\tau}$, using
 		\ref{eq:ubabqunique}, and proceeding as in the proof of Theorem \ref{thm:abstab} including the fact that the locations $\xi_1, \cdots, \xi_4$ are diverse w.r.t $(T+1)B$, for large enough
 		$\sigma$ we obtain
 		\begin{align}\label{pre-uniqueness 1}
 		   \sigma \|[\overline a, \overline b]\|^2_{0, \sigma, \real^3 \times [0, T]} 
 		   ~ \le ~ C_1 \left ( 
 		   \|[\overline a, \overline b, \overline q]\|^2_{0, \sigma, \real^3 \times [0, T]} 
 		   \right ).
 		\end{align}
 		Note the $\qbar$ term on the RHS of \eqref{pre-uniqueness 1}. This was absent 
 		from the RHS of the similar inequality, when proving Theorem \ref{thm:abstab}, because $\qbar=0$ for
 		Theorem \ref{thm:abstab}. 
 		
 	Next we estimate the norm of $\qbar$ using the data coming from $v, \vacu$. 
 		Taking the differences of 
 		\eqref{ch BVP for delta 1}, \eqref{ch BVP for delta 2} for the 
 		coefficients $(a, b, q)$ and $(\acute a, \acute b, \acute q)$, 
 		we have
 		\begin{align}
 		\label{pre-uniqueness delta 1}   
 		\mc{L} \overline v_4 =  2 \overline a \,\vacu_{4t} - 
 		2 \overline b \cdot \nabla \acute v_4 - \overline q \acute v_4    , & \qd \tn{in } Q_{\xi_4, \tau}, 
 		\\
 		\label{pre-uniqueness delta 2}   
 		2(\dd_t + \theta_4 \cdot \nabla - (a + \theta_4 \cdot b) + r^{-1}) \overline v =  \acute{\mc{L}}\acute \alpha_4 - \mc{L} \alpha_4 + 2 ( \overline a + \theta_4 \cdot \overline b) \acute v_4, 
 		& \qd \tn{on } C_{\xi_4, \tau}.
 		\end{align}
 		Since $ a + \theta_4 \cdot b = \acute a + \theta_4 \cdot \acute b = 0$, we have 
 		\begin{align*}
 		    \alpha_4 (x,t) = \acute\alpha_4(x,t)= \f{1}{|x - \xi_4|},
 		\end{align*}
 		hence, for $x \neq \xi_4$,
 		\begin{align*} 
 		    \acute{\mc{L}}\acute \alpha_4 - \mc{L} \alpha_4 = -\left(\mc{L} - \acute{\mc{L}}\right)\alpha_4 = \f{ \overline b(x,t) \cdot (x - \xi_4)}{|x - \xi_4|^3} - \f{\overline q(x,t)} {|x - \xi_4|},
 		\end{align*}
 		so \eqref{pre-uniqueness delta 2} gives us
 		\begin{align}\label{pre-uniqueness reln 2}
 		 2 \left(\dd_t + \theta_4 \cdot \nabla - (a + \theta_4 \cdot b) + r^{-1} \right)\overline v = 
 		 \f{2 \overline \theta_4 \cdot \bbar}{|x - \xi_4|^2} - \f{\overline q (x,t)}{|x - \xi_4|},
 		\qquad  \text{~on } C_{\xi_4, \tau}.
 		\end{align}
 		
 	From (\ref{pre-uniqueness reln 2}), noting that $\{\xi_4 \} \times (-\infty, T]$ does not intersect the supports of $\abar, \bbar, \qbar$ and $\vbar$, we have
 		\begin{align*}
 		|\qbar| \, \cleq \, |\bbar| + |\nabla_C \, \vbar| + |\vbar|,
 		\qquad \text{on } C_{\xi_4,\tau}
 		\end{align*}
 		where $\nabla_C$ is the gradient on $C_{\xi_4, \tau}$ and the constant is dependent only
 		on $T, \xi_4$ and $\|[a,b, \vacu_4]\|_{C^0}$, hence only on $T, M$ and $\vert \xi_4 \vert$.
 		Therefore
 		\[
 		\|\qbar\|_{0,\sigma, C_{\xi_4,\tau}}
 		\cleq \,
 		\|[\abar, \bbar] \|_{0,\sigma, C_{\xi_4,\tau}} + \|\vbar\|_{1, \sigma, C_{\xi_4,\tau}}.
 		\]
 		Also,  from (\ref{pre-uniqueness delta 1}), we have
 		\[
 		\|\mc{L} \vbar_4 \|_{0,\sigma, Q_{\xi_4,\tau}} \cleq \, 
 		\|[\abar, \bbar, \qbar]\|_{0,\sigma, Q_{\xi_4, \tau}}
 		\]
 		with the constant dependent only on $T$ and $\|[\vacu_4]\|_{C^1}$, hence dependent only on $T$ and $M$. 
 		Using these observations in Proposition \ref{prop:carleman} applied to
 		$\overline v$ on the region $Q_{\xi_4, \tau}$ and noting that $\vbar, \vbar_t$ are zero on
 		$H_{\xi_4, \tau}$, we obtain
 		\begin{align*}
 		    \sigma\|\overline q\|^2_{0, \sigma, C_{\xi_4, \tau}} 
 		    \cleq \|[\abar, \bbar, \overline q\|^2_{0, \sigma, Q_{\xi_4, \tau}}
 		    + \sigma \|[\overline a, \overline b]\|^2_{0,\sigma, C_{\xi_4,\tau}},
 		    \qquad \forall \tau \in [\tau_{min}, \tau_{max}],
 		\end{align*}
 		for $\sigma$ large enough. 
 		Integrating this inequality w.r.t $\tau$, over the interval $[\tau_{min},
 		\tau_{max}]$, using integral relations similar to the one in the proof of Theorem \ref{thm:qstab}, we obtain
 		\begin{align*}
 		    \sigma \|\overline q\|^2_{\real^3 \times [0,T]}
 		    \le 
 		    C_2 \left ( \|\overline q\|^2_{\real^3 \times [0,T]}
 		    + \sigma \|[\overline a, \overline b]\|^2_{\real^3 \times [0,T]} \right ),
 		\end{align*}
 		for $\sigma$ large enough.
 		So adding to this a $C_2+1$ multiple of \eqref{pre-uniqueness 1} 
 		we obtain
 		\begin{align*}
 		    \sigma \|[\overline a, \overline b, \overline q]\|^2_{0, \sigma, \real^3 \times [0, T]} 
 		   \, \cleq \, \|[\overline a, \overline b, \overline q]\|^2_{0, \sigma, \real^3 \times [0, T]},
 		\end{align*}
 		for $\sigma$ large enough. Hence, choosing $\sigma$ large enough, we obtain $\abar=0, \bbar=0, \qbar=0$, proving the theorem.


\section{Proof of Theorem \ref{thm:abcstab}}


{
In Theorem \ref{thm:abcunique} our aim was uniqueness in the recovery of curl$(a,b)$ and $c$ and that followed from
uniqueness in the recovery of a gauge equivalent $a,b,c$ for which $a+ \theta \cdot b=0$. It did not require estimating
first order derivatives of the gauge equivalent $a,b$. The situation in the proof of Theorem \ref{thm:abcstab} is different as a stability estimate is desired in the recovery of curl$(a,b)$ and $c$; so we must estimate first order derivatives of the gauge equivalent $a,b,c$.

We estimate the first order derivatives of $\abar, \bbar$ using the PDEs and the boundary conditions for $\ubar$ and $\vbar$. We note that
\[
\mc{L} - \mc{\Lacu}
= \Box - 2 \abar \pa_t + 2 b \cdot \nabla + (a^2 - b^2 - \aacu^2 + \bacu^2) + (\cbar - \abar_t + \nabla \cdot \bbar).
\]
The $-\abar_t + \nabla \cdot \bbar$ term prevents us from getting useful estimates on $\abar, \bbar$ because this term results in 
$\|[\abar, \bbar]\|_0$ being estimated by $\|[\abar, \bbar]\|_1$ and $\|[\abar, \bbar]\|_1$ estimated by
$\|[\abar, \bbar]\|_2$. The way out is to replace $a,b$ (and $\aacu, \bacu$) by a gauge equivalent $a,b$ such that 
$c- a_t + \nabla \cdot b =0$. Such a gauge equivalent $a,b$ may be constructed with the gauge $\psi$ in the statement of the theorem, but the data for this gauge equivalent $a,b,c$ requires the original data and the value of $\psi$ on $t=T$. That is why the statement of Theorem
\ref{thm:abcstab} has the $\psi$ terms on the RHS.
Note that when we have estimated the first order derivatives of $\abar, \bbar$, we automatically get an estimate for $\cbar$ since
$\cbar = \abar_t - \nabla \cdot \bbar$.

The PDE and the boundary condition for $\ubar$ are used to estimate $\abar, \bbar$, as done in the proof of Theorem \ref{thm:abstab}. We also differentiate this PDE in directions tangential to the cone $C_{\xi, \tau}$ to estimate these directional derivatives of $\abar, \bbar$. However, to estimate the derivative of $\abar, \bbar$ in the remaining direction, the normal to $C_{\xi, \tau}$, we use the PDE and the boundary condition for $\vbar$ equation, and a decomposition of 
$\mc{L} \alpha$ as the sum of the normal (to $C_{\xi, \tau}$) derivative of $(a+\theta \cdot b)$ and derivatives of $\alpha$ in directions tangential to $C_{\xi, \tau}$ - see (\ref{eq:Lalphaatemp}).

We now proceed with the proof of the theorem.

}

The $u, \uacu, v, \vacu$ are the solutions guaranteed by Propositions \ref{Heaviside} and
 	    \ref{Delta} for the coefficients $a,b,c$ and  $\aacu, \bacu, \cacu$.
 	    Since $a,b,c$ are supported in $\Bbar \times [0,T]$ and 
 		$\xi \in \real^3 \setminus (T+1)\Bbar$, one may check that for a fixed $\xi \in (T+1)B$, 
 		the values of $u, u_t, u_{tt}, v, v_t$ 
 		and $\uacu, \uacu_t, \uacu_{tt}, \vacu, \vacu_t$ on $\real^3 \times \{t \smeq T \}$ are zero for $\tau>T+1 - |\xi|$
 		and do not change as $\tau$ varies over $\tau \in (-\infty, -1 - |\xi|]$. Hence, in the statement of Theorem \ref{thm:abcstab}, the $\tau$ integrals may be replaced by $\tau$ integrals over any finite interval that contains $[-1-|\xi_i|, T+1 - |\xi_i|]$.
 		
 		We note that $\psi(x,t)$ is a smooth function on $\real^3 \times \real$ and its support intersects
 		$\real^3 \times (-\infty,T]$ in a region contained in $(T+1)\Bbar \times (-\infty,T]$. 		
 		
 		We recall from the first section that, if $U$ and $V$ are the solutions, corresponding to the 
 		$a, b, c$, guaranteed by Propositions \ref{Heaviside} and \ref{Delta}, 
 		then $e^{\psi} U$ and $e^{\psi} V$ are the solutions, corresponding to the coefficients 
 		$a + \psi_t, b + \nabla \psi, c$, guaranteed by Propositions \ref{Heaviside} and \ref{Delta}. Further
 		one may verify that $[a,b]$ and $[a + \psi_t, b + \nabla \psi]$ have the same curl and
 		\[
 		c - \left(a + \psi_t\right)_t + \nabla \cdot \left( b + \nabla \psi \right) = 0.
 		\]
 		We also observe that
 		the intersection of the supports of $a + \psi_t, b + \nabla \psi, c$, with
 		$\real^3 \times (-\infty, T]$ are contained in $(T+1)\Bbar \times (-\infty, T]$. In particular,
 		$\{\xi_i\} \times (-\infty,T]$ does not intersect the supports of $a + \psi_t, b + \nabla \psi, c$
 		so Propositions \ref{Heaviside} and \ref{Delta} apply even to these modified $a,b,c$.
 		
 		For bounded functions $\psi,\acute \psi,w, \acute w$, we have
 		\begin{align*}
 		    \left| e^{\psi} w - e^{\acute \psi} \acute w \right| \le & \left|e^{\psi} (w - \acute w)\right| + \left|\left(e^{\psi} - e^{\acute \psi}\right) \acute w\right| \\
 		    & \le C \left( |w - \acute w| + \left|e^{\psi} - e^{\acute \psi}\right|\right) \\
 		    & \le C \left( |w - \acute w| + |\psi - \acute \psi| \right),
 		\end{align*}
 		and one has similar estimates for the first and second order derivatives of $e^\psi w, e^{\psiacu} \wacu$ also. 
 		
 		Keeping the above observations in mind, it is enough to prove 
 		Theorem \ref{thm:abcstab} for $a,b,c$ and 
 		$\aacu, \bacu, \cacu$ which are supported in $(T+1) \Bbar \times (-\infty,T]$ for which
 		\begin{align*}
 		    c- a_t + \nabla \cdot b = 0, \qd \cacu - \aacu_t + \nabla \cdot \acute b = 0,
 		\end{align*}
 		and the $\tau$ integrals in the statement of Theorem \ref{thm:abcstab} are over $[\tau_{min}, \tau_{max}]$ where
 		\[
 		\tau_{min} := \text{min} \{ - (|\xi_i| + T+1): i=1,2,3,4 \},
 		\qquad
 		\tau_{max} := \text{max} \{ 2T+1 - |\xi_i|: i=1,2,3,4\}.
 		\]

 		For use below, observe that
 		the union of the $C_{\xi,\tau}$, as $\tau$ varies
 		over $[\tau_{min}, \tau_{max}]$, contains $(T+1)\Bbar \times [0,T]$, so, in particular, it contains
 		the intersection of $\real^3 \times (-\infty,T]$ with the supports of the modified 
 		$a,b,c, \aacu, \bacu, \cacu$.
 		
 		Note that $\psi=0$ and $\psiacu=0$ for these modified
 		$a,b,c$ and $\aacu, \bacu, \cacu$ and the operators $\mc{L}$ and $\mc{\acute L}$ become 
 		\begin{align*}
 		    \mc{L} = \square - 2 a \dd_t + 2 b \cdot \nabla + a^2 - b^2, \qd
 		    \mc{\acute L} = \square - 2 \acute a \dd_t + 2 \acute b\cdot \nabla  + \acute a^2 - \acute b^2.
 		\end{align*}

 		To keep the expressions simple, while we work with
 		a fixed but arbitrary $\xi \in \real^3 \setminus (T+1)\Bbar$ and $\tau \in [\tau_{min},
 		\tau_{max}]$, we 
 		suspend showing
 		the dependence on $\xi, \tau$. We write $Q_{\xi,\tau}$,
 		$H_{\xi,\tau}$ and $C_{\xi,\tau}$ as $Q,H$ and $C$. Further, 
 		we define
 		\[
 		D = \nabla_{x,t}, \qquad r(x) = |x- \xi|, \qquad \theta(x) = \frac{x-\xi}{|x-\xi|},
 		\qquad x \neq \xi.
 		\]
 		and we write them as $r,  \theta$.
 		
 		We have
 		\begin{align*}
 		(a^2 - |b|^2) - (\aacu^2 - |\bacu|^2)
 		= (a+\aacu)(a-\aacu) - (b + \bacu) \cdot (b - \bacu)
 		= (a+\aacu) \abar - (b+ \bacu) \cdot \bbar.
 		\end{align*}
 		Hence $\ubar := u - \uacu$ is a solution of the characteristic BVP
 		\begin{align}
 		 \label{gen diff eqn for u 1}   \mc{L} \overline u = 
 		 2 \overline a \acute u_t - 2 \overline b \cdot \nabla \acute u +  2 r^{-2} \theta \cdot \bbar  +\left((b + \acute b) \overline b - (a + \acute a) \overline a \right) \left(\acute u + r^{-1}
 		 \right),
 		 & \qd \tn{ in } Q, \\
 		 \label{gen diff eqn for u 2}  \overline u = \alpha - \acute \alpha, & \qd \tn{ on } C.
 		\end{align}
 		So applying Proposition \ref{prop:carleman} to $\ubar$ on $Q$, we obtain
 		\begin{align} \label{u estimate}
 		    \sigma \left( \|\overline u\|^2_{1, \sigma, Q} + \|\alpha - \acute \alpha\|^2_{1, \sigma, C} \right) \cleq
 		    \, \|[\overline a, \overline b]\|^2_{0, \sigma, Q} 
 		    + \|\overline u\|^2_{1, \sigma, H} + \|\dd_t \overline 
 		    u\|^2_{0, \sigma, H} .
 		\end{align}
 	
 		Next, we obtain higher order estimates on $\ubar$ by differentiating 
 		(\ref{gen diff eqn for u 1}) in directions tangential to $C$.
 		If we write $x = (x^1, x^2, x^3)$ and $\xi = (\xi^1, \xi^2, \xi^3)$,
 		the vector fields 
 		\[
 		\dd_t + \theta\cdot \nabla, \qquad 
 		\Omega_{lm} = (x^l - \xi^l) \partial_m - (x^m - \xi^m) \partial_l, ~~l,m= 1,2,3,
 		\]
 		span the tangent space to $C$ at any point on $C$. Differentiating
 		(\ref{gen diff eqn for u 1}), (\ref{gen diff eqn for u 2})
 		by $ \Omega_{lm}$ we obtain
 		\begin{align}
 		   \nn  \mc{L}\left(\Omega_{lm}\overline u\right) = &\  \Omega_{lm}\left( 2 \overline a \acute u_t - 2 \overline b \cdot \nabla \acute v + 2 r^{-2} \theta \cdot \bbar 
 		   +\left((b + \acute b) \overline b - (a + \acute a) \overline a \right) \left(\acute u + r^{-1} \right)\right) \\
 		   \label{radial der of u 1} & \qd + [\mc{L}, \Omega_{lm}]\overline u, \qd \tn{ in } Q, \\
 		   \label{radial der of u 2}   \Omega_{lm} (\overline u) = &\  \Omega_{lm} \left(\alpha - \acute \alpha \right), 
 		   \qd \tn{ on } C.
 		\end{align}
 		 Since principal part of $\mc{L}$ is a constant 
 		 coefficient operator, the operator $[\Lcal, \Omega_{lm}]$ is a first order operator, hence (\ref{radial der of u 1}) implies
 		\begin{align}\label{commutator 1}
 		   \left| \mc{L}\left(\Omega_{lm} \overline u\right) \right| 
 		   \cleq  \left|[\overline a, \overline b, D \overline a, D \overline b]\right| + 
 		   \left|[\overline u, D \overline u]\right|,
 		   \qquad \text{on } Q.
 		\end{align}
 		Hence, using \eqref{radial der of u 2} and \eqref{commutator 1}
 		and Proposition \ref{prop:carleman} applied to $\Omega_{lm}\overline u$ on $Q$, we obtain
 		\begin{align}
 		 \sigma \|\Omega_{lm} (\alpha - \acute \alpha)\|^2_{1, \sigma, C}
 		 \, \cleq \, & 
 		 \|[\overline a, \overline b, D \abar, D \bbar]\|^2_{0, \sigma, Q} 
 		 + \|[\overline u, D \ubar]\|^2_{0, \sigma, Q}  
 		 \nn
 		 \\
 		 & \qquad + \sigma 
 		 \|\Omega_{lm} \overline u\|^2_{1, \sigma, H} 
 		  + \sigma \|\dd_t \left(\Omega_{lm} \overline u\right)\|^2_{0, \sigma, H}.
 		  \label{derivative estimate 1} 
 		\end{align}
 		
 		Using a similar argument for the vector field 
 		$\dd_t + \theta \cdot \nabla$, we obtain
 		\begin{align}
 		 \sigma \|(\dd_t + \theta \cdot \nabla)(\alpha - \acute \alpha)\|^2_{1, \sigma, C}
 	\,	 \cleq \, 
 	\|[\overline a, \overline b, D \abar, & D \bbar]\|^2_{0, \sigma, Q} 
 		 + \|[\overline u, D \ubar]\|^2_{0, \sigma, Q}  
 		 \nn
 		 \\
 		   & + \sigma \|[ \overline u_t, \nabla \overline u]\|^2_{1, \sigma, H}
 		   + \sigma\| \overline u_{tt}\|^2_{0, \sigma, H}.
 		  \label{derivative estimate 2}
 		\end{align}
 
 Combining (\ref{u estimate}), (\ref{derivative estimate 1}) and (\ref{derivative estimate 2}) we obtain
 \begin{align}
    \sigma  \| [ \alpha - \alphaacu, & (\pa_t + \theta \cdot \nabla)(\alpha - \alphaacu), 
     \Omega_{lm} (\alpha - \alphaacu)] \|^2_{1, \sigma, C} 
     \nn
     \\
    &  \cleq \, 
      \|[\abar, \bbar, D \abar,  D \bbar]\|^2_{1,\sigma,Q} 
     +  \sigma (\|\ubar\|^2_{2, \sigma, H} + \| \ubar_t\|^{{2}}_{1,\sigma,H} + \| \ubar_{tt}\|^{{2}}_{0,\sigma,H} ),
     \label{eq:thm17abbb}
 \end{align}
 for large enough $\sigma$, where the $\|\ubar, D \ubar\|_{0,\sigma,Q}$ term on the 
 RHS of (\ref{derivative estimate 1}) is absorbed by the $\sigma \|\ubar\|_{1,\sigma,Q}$ term on the LHS
 of \eqref{u estimate}.

Below, we will need the observations that  $\alpha - \alpha'=0$ in a neighborhood of 
 		$\{\xi_i\} \times (-\infty, T]$ and $\alpha$ and its derivatives are bounded on the supports of $\abar, \bbar$. Further, $\alpha$ is positive and bounded away from zero on 
 		$(T+1) \Bbar \times (-\infty,T]$.

We use (\ref{eq:thm17abbb}) and the relation between $\alpha - \alphaacu$ and 
$\abar + \theta \cdot \bbar$  to obtain estimates for $\abar + \theta \cdot \bbar$.
Using (\ref{transport eqn}), we observe 
 		\begin{align}
 		    \left(\dd_t + \theta\cdot \nabla\right)(\alpha - \acute \alpha) 
 		    & = (a+\theta \cdot b) \alpha - (\aacu +\theta \cdot \bacu) \alphaacu 
 		    -r^{-1}(\alpha - \acute \alpha)
 		    \nn
 		    \\
 		    & = \alpha (\overline a + \theta \cdot\overline b) + (a + \theta \cdot b - r^{-1})(\alpha - \acute \alpha).
 		    \label{eq:thm17a}
 		\end{align}
 This implies
 		\beqn
 		\|\abar + \theta \cdot \bbar\|_{0,\sigma,C}^2 \,
 		\cleq \,  \|(\pa_t + \theta \cdot \nabla)(\alpha - \alphaacu)\|^2_{0, \sigma, C} + \|\alpha - \alphaacu\|_{0, \sigma, C}^2.
 		\label{eq:thm17d}
 		\eeqn
 		
Next, differentiating (\ref{eq:thm17a}) w.r.t $\pa_t + \theta \cdot \nabla$ we obtain
 		\begin{align}
 		(\pa_t + \theta \cdot \nabla)^2 (\alpha - \alphaacu) 
 		& = \alpha (\pa_t + \theta \cdot \nabla) (\abar + \theta \cdot \bbar) + f \, (\abar + \theta \cdot \bbar)
 		+ g \, (\pa_t + \theta \cdot \nabla) (\alpha - \alphaacu) 
 		\nn\\
 		& \qquad + h \, (\alpha - \alphaacu)
 		\label{eq:thm17b}
 		\end{align}
 		for some bounded functions $f,g,h$.
 		Similarly, differentiating (\ref{eq:thm17a}) w.r.t $\Omega_{lm}$ we obtain
 		\begin{align}
 		\Omega_{lm} (\pa_t + \theta \cdot \nabla) (\alpha - \alphaacu)
 		& = \alpha \, \Omega_{lm} (\abar + \theta \cdot \bbar) + f \, (\abar + \theta \cdot \bbar)
 		+ g \, \Omega_{lm} (\alpha - \alphaacu) 
 		 + h \, (\alpha - \alphaacu)
 		 \label{eq:thm17c}
 		\end{align}
 		for some bounded functions $f,g,h$.
 		
 		Using (\ref{eq:thm17b}), (\ref{eq:thm17c}), we obtain
 		\begin{align*}
 		  \|(\dd_t & + \theta\cdot \nabla)  (\overline a+ 
 		  \theta \cdot \overline b)\|^2_{0, \sigma, C} 
 		  + \sum_{l,m= 1}^{3}\|\Omega_{lm}(\overline a+ \theta \cdot \overline b)\|^2_{0, \sigma, C} \\
 		 & \cleq \, \|\abar + \theta \cdot \bbar \|^2_{0,\sigma, C}
 		 + \|(\dd_t + \theta \cdot \nabla) \left(\alpha - \acute \alpha \right)\|^2_{1, \sigma, C} + \sum_{l,m=1}^{3}\|\Omega_{lm} \left(\alpha - \acute \alpha \right)\|^2_{1, \sigma, C} 
 		 + \|\alpha - \acute \alpha\|^2_{1, \sigma, C} .
 		\end{align*}
 		If we use (\ref{eq:thm17d}) in this, we obtain
 		\begin{align}
 		  \nn  \| \abar +  \theta \cdot \bbar & \|^2_{0,\sigma, C}
 		  + \|(\dd_t  + \theta\cdot \nabla)  (\overline a+ 
 		  \theta \cdot \overline b)\|^2_{0, \sigma, C} 
 		  + \sum_{l,m= 1}^{3}\|\Omega_{lm}(\overline a+ \theta \cdot \overline b)\|^2_{0, \sigma, C} 
 		  \\
 		 & \cleq \,
 		 \|(\dd_t + \theta \cdot \nabla) \left(\alpha - \acute \alpha \right)\|^2_{1, \sigma, C} + \sum_{l,m=1}^{3}\|\Omega_{lm} \left(\alpha - \acute \alpha \right)\|^2_{1, \sigma, C} 
 		 + \|\alpha - \acute \alpha\|^2_{1, \sigma, C}.
 		 \label{trans eqn for diff of alpha} 
 		\end{align}
 	So combining \eqref{eq:thm17abbb} and \eqref{trans eqn for diff of alpha} we obtain
 	\begin{align}
  \| \abar +  \theta \cdot \bbar  \|^2_{0,\sigma, C}
 + \|(\dd_t  + \theta\cdot \nabla)  & (\overline a+ \theta \cdot \overline b)\|^2_{0, \sigma, C} 
 		  + \sum_{l,m= 1}^{3}\|\Omega_{lm}(\overline a+ \theta \cdot \overline b)\|^2_{0, \sigma, C} 
 		  \nn
 		  \\
 		  & \cleq \; 
 		  \frac{1}{\sigma}\|[\abar, \bbar, D \abar, D \bbar] \|^{{2}}_{0,\sigma, Q} 
 		  + \|\ubar\|^2_{2, \sigma, H} + \| \ubar_t\|^{{2}}_{1,\sigma,H} + \| \ubar_{tt}\|^{{2}}_{0,\sigma,H}.
 		 \label{eq:thm17pqr}
 	\end{align}

 		We now obtain estimates of a derivative of $\abar + \theta \cdot \bbar$ in a direction not 
 		tangential to $C$, using the $V$ solution. We note that $\vbar:= v - \vacu$ is a solution of
 		the characteristic BVP
 		\begin{align}
 		   \label{gen diff eqn for v 1}  \mc{L} \overline v = 2 \overline a \acute v_t - 2 \overline b\cdot \nabla \acute v + \left((b + \acute b) \overline b - (a + \acute a) \overline a \right) \acute v, 
 		   & \qd \tn{ in } Q,  \\ 
 		  \label{gen diff eqn for v 2}  
 		  2 \left(\dd_t + \theta \cdot \nabla - (\overline a + \theta \cdot \overline b) + r^{-1}\right) \overline v= \acute{\mc{L}} \acute \alpha - \mc{L} \alpha + 2 (\overline a + \theta \cdot \overline b) \acute v, 
 		  & \qd \tn{ on } C.
 		\end{align}
 		So applying Proposition \ref{prop:carleman} to $\overline v$ on the region $Q$ and using 
 		\eqref{gen diff eqn for v 1}, we obtain
 		\begin{align}\nn
 		   \sigma \|\overline v\|^2_{1, \sigma, C} ~ \cleq ~
 		   \|[\overline a, \overline b]\|^2_{0,\sigma,Q} + 
 		   \sigma \|\overline v\|^2_{1, \sigma, H} + \sigma \| \overline v_t \|^2_{0, \sigma, H},
 		\end{align} 
 		so, in particular,
 		\begin{align}\label{intermediate v estimate 1}
 		  \| \vbar \|_{0,\sigma, C}^2 + \| ( \dd_t + \theta \cdot \nabla) 
 		    \overline v\|^2_{0, \sigma, C}
 		    ~\cleq ~
 		    \frac{1}{\sigma} \|[\overline a, \overline b]\|^2_{0,\sigma,Q} + 
 		    \|\overline v\|^2_{1, \sigma, H} + \|\dd_t \overline v\|^2_{0, \sigma, H}.
 		\end{align}
 		From \eqref{gen diff eqn for v 2} we have
 		\begin{align*}
 		   |(\dd_t + \theta \cdot \nabla) \overline v|
 		   + |\vbar| \, \cgeq 
 		   \, |\acute{\mc{L}} \acute \alpha - \mc{L} \alpha| - |[\overline a, \overline b]| ,
 		    \qquad \text{on } C.
 		\end{align*}
 		So using this in \eqref{intermediate v estimate 1} we obtain
 		\begin{align}
 		    \| \mc{L} \alpha -  \Lacu \alphaacu \|^2_{0,\sigma,C}
 		    \, \cleq \,
 		    \|[\abar, \bbar]\|^2_{0, \sigma, C} + \frac{1}{\sigma} \|[\overline a, \overline b]\|^2_{0,\sigma,Q} + 
 		    \|\overline v\|^2_{1, \sigma, H} + \|\dd_t \overline v\|^2_{0, \sigma, H}.
 		    \label{eq:thm17rst}
 		\end{align}
 
 {
		We claim
		\beqn
		\mc{L} \alpha = \alpha \left(\dd_t - \theta \cdot \nabla \right) (a + \theta \cdot b) - 
		\Delta_S \alpha + 2 b^\perp \cdot \nabla \alpha 
		- ( |b^\perp|^2 + 2 r^{-1} \theta \cdot b) \alpha,
		\label{eq:Lalphaatemp}
		\eeqn
    where $\Delta_S$ denotes the spherical Laplacian and $b^\perp = b - (b\cdot\theta)\theta$. This identity expresses the normal (to $C_{\xi,\tau}$) derivative of $a+\theta b$ in terms of $\mc{L} \alpha$ and the tangential (to $C_{\xi,\tau}$) derivatives of $\alpha$. We restate this relation 
    as Lemma \ref{decomposition} and its proof is given at the end of the proof of Theorem \ref{thm:abcstab}
    }
 		
 		Using \eqref{eq:Lalphaatemp} for $\mc{L} \alpha$ and $\Lacu \alphaacu$ and subtracting those identities, we obtain
 		\begin{align*}
 		    \mc{L}\alpha - \acute{\mc{L}}\acute \alpha =\  
 		    & \alpha ( \dd_t - \theta \cdot \nabla)( \overline a + \theta \cdot \overline b) + (\alpha - \acute \alpha)(\dd_t - \theta \cdot \nabla)(\acute a + \theta \cdot \acute b) 
 		     - \Delta_S (\alpha - \acute \alpha)  
 		     \\
 		     & \qquad + 2 (b^\perp \cdot \nabla) (\alpha - \alpha')
 		     -2 \bbar^\perp \cdot \nabla \alphaacu
 		     - (|b^\perp|^2 + 2 r^{-1} \theta \cdot b)(\alpha - \acute \alpha) 
 		     \\
 		      & \qquad   - ((b^\perp + \bacu^\perp) \cdot \overline b^{\perp} + 2 r^{-1} \theta \cdot \bbar) \alphaacu,
 		\end{align*}
 		implying
 		\begin{align*}
 		  |\mc{L}\alpha - \acute{\mc{L}}\acute \alpha|  
 		 \,  \cgeq  \,|( \dd_t -  \theta \cdot \nabla)( \overline a + \theta \cdot \overline b) | - 
 		 |\alpha - \acute \alpha| - |(b^\perp \cdot \nabla)(\alpha - \alphaacu)|
 		  - | \Delta_{S}(\alpha - \acute\alpha)| - |\overline b|,
 		\end{align*}
 		where we have used the fact that $\alpha$ has a positive lower bound on $C$. Hence using this in
 		\eqref{eq:thm17rst} we obtain
 		\begin{align}
 		  \nn  \| (\dd_t - \theta\cdot \nabla)(\overline a + \theta \cdot \overline b)\|^2_{0, \sigma, C}
 		 ~ \cleq ~\ &  \frac{1}{\sigma} \|[\overline a, \overline b]\|^2_{0, \sigma,Q}
 		  + \|[\overline a, \overline b]\|^2_{0, \sigma, C}
 		  + \|\overline v\|^2_{1, \sigma,H} + \|\overline v_t\|^2_{0, \sigma, H}
 		  \nn
 		  \\
 		 & \qquad + \ \|[\alpha - \acute \alpha, (b^\perp \cdot \nabla)(\alpha - \alphaacu), 
 		 \Delta_{S} (\alpha -\acute \alpha)] \|^2_{0, \sigma, C}.
 		 \label{eq:thm17uvw}
 		\end{align}

 		Now we combine the estimates obtained from the $U$ and the $V$ solutions. Now
 		$b^\perp := b - (\theta \cdot b) \theta$ is perpendicular to $\theta$ (the radial direction) and has no component in the $t$ axis direction, so $b^\perp \cdot \nabla$ is in the span of the $\Omega_{lm}$. Further, $\Delta_S$ is a second order operator made up of $\Omega_{lm}$. That is, $\Delta_S = \f{1}{2r^2}\sum\limits_{l,m=1}^3\Omega_{lm}^2$.
 		Hence
 		\[
 		\|[(b^\perp \cdot \nabla)(\alpha - \alphaacu), \Delta_S(\alpha - \alphaacu)]\|_{0,\sigma,C} 
 		\ \cleq  \ 
 		\sum_{l,m=1}^3 \| \Omega_{lm} (\alpha - \alphaacu) \|_{1, \sigma, C},
 		\]
 		so using \eqref{eq:thm17abbb} in \eqref{eq:thm17uvw} we obtain
 		\begin{align}
 		\| (\dd_t - \theta\cdot \nabla)(\overline a + \theta \cdot \overline b)\|^2_{0, \sigma, C}
 		\,  \cleq \, \frac{1}{\sigma} & \|[\abar, \bbar, D \abar, D \bbar]\|^2_{0,\sigma,Q} + \|[\abar, \bbar] \|^2_{0,\sigma,C}
 		+ \|\ubar\|^2_{2,\sigma,H} 
 		\nn\\
 		& \qquad + \| [\ubar_t, \vbar]\|^2_{1, \sigma, H} + \|[\ubar_{tt}, \vbar_t]\|^2_{0,\sigma, H}.
 		\label{eq:thm17xyz}
 		\end{align}
 		
 		Now $\theta \cdot \nabla$ represent the radial derivative $\pa_r$ in $\real^3$. Further
 		$\pa_t - \pa_r, \pa_t + \pa_r, \Omega_{lm}$, $l,m=1,2,3$, span the tangent space to $\real^4$. 
 		Hence \eqref{eq:thm17pqr} and \eqref{eq:thm17xyz} give us
 		\begin{align}
 		\| D (\abar + \theta \cdot \bbar)\|^2_{0,\sigma,C}
 		\, & \cleq \, \frac{1}{\sigma} \|[\abar, \bbar, D \abar, D \bbar]\|^2_{0,\sigma,Q} 
 		+ \|[\abar, \bbar] \|^2_{0,\sigma,C}
 		+ \|\ubar\|^2_{2,\sigma,H} 
 		\nn\\
 		& \qquad + \| [\ubar_t, \vbar]\|^2_{1, \sigma, H} + \|[\ubar_{tt}, \vbar_t]\|^2_{0,\sigma, H}.
 		\label{eq:thm17xab}
 		\end{align}
 		Also, from \eqref{eq:thm17pqr}, we can extract
 		\begin{align}
 			\| \abar + \theta \cdot \bbar \|^2_{0,\sigma,C}
 		\,  \cleq \, \frac{1}{\sigma} \|[\abar, \bbar, D \abar, D \bbar]\|^2_{0,\sigma,Q} 
 		+ \|\ubar\|^2_{2,\sigma,H} 
 		 + \| [\ubar_t, \vbar]\|^2_{1, \sigma, H} + \|[\ubar_{tt}, \vbar_t]\|^2_{0,\sigma, H}.
 		\label{eq:thm17lmn}
 		\end{align}
 		Note that $\|[\abar, \bbar]\|_{0,\sigma,C}$ term is absent from the RHS of \eqref{eq:thm17lmn}. 
 		This will be significant.
 		
 		We integrate the last two inequalities w.r.t $\tau$, over the interval $[\tau_{min}, \tau_{max}]$. Using integral relations similar to those used in the proof of Theorem \ref{thm:qstab}, we obtain
 		\begin{align}
 		 \|D (\abar + \theta \cdot \bbar) \|^2_{0, \sigma, \real^3 \times [0,T]}
 		 \, & \cleq \, \frac{1}{\sigma} \| [\abar, \bbar, D \abar, D \bbar]\|^{{2}}_{0,\sigma, \real^3 \times [0,T]}
 		 +  \| [\abar, \bbar]\|^{{2}}_{0,\sigma, \real^3 \times [0,T]}
 		 \nn
 		 \\
 		 & ~~ + \int_{\tau_{min}}^{\tau_{max}} \left ( \|\ubar\|^2_{2,\sigma,H} 
 		 + \| [\ubar_t, \vbar]\|^2_{1, \sigma, H} + \|[\ubar_{tt}, \vbar_t]\|^2_{0,\sigma, H} \right ) 
 		 \, d \tau 
 		 \label{eq:thm17xxx}
 		\end{align}
 		and
 		\begin{align}
 		 \|\abar + \theta \cdot \bbar \|^2_{0, \sigma, \real^3 \times [0,T]}
 		 \, & \cleq \, \frac{1}{\sigma} \| [\abar, \bbar, D \abar, 
 		 D \bbar ]\|^{{2}}_{0,\sigma, \real^3 \times [0,T]}
 		 \nn
 		 \\
  & \qquad + \int_{\tau_{min}}^{\tau_{max}} \left ( \|\ubar\|^2_{2,\sigma,H} 
 		 + \| [\ubar_t, \vbar]\|^2_{1, \sigma, H} + \|[\ubar_{tt}, \vbar_t]\|^2_{0,\sigma, H} \right ) \, d \tau.
 		 \label{eq:thm17yyy}
 		\end{align}
 		
 	We have (\ref{eq:thm17yyy}) for $\xi= \xi_i$, $i=1,2,3,4$ and the locations
 	$\xi_1, \cdots, \xi_4$ are diverse with respect to $(T+1)B$, hence
 	\begin{align}
 		 \|[\abar, \bbar] \|^2_{0, \sigma, \real^3 \times [0,T]}
 		 \, & \cleq \, \frac{1}{\sigma} \| [\abar, \bbar, D \abar, 
 		 D \bbar]\|^{{2}}_{0,\sigma, \real^3 \times [0,T]}
 		 \nn
 		 \\
  & \qquad + \int_{\tau_{min}}^{\tau_{max}} 
  \sum_{i=1}^4 \left ( \|\ubar\|^2_{2,\sigma,H} 
 		 + \| [\ubar_t, \vbar]\|^2_{1, \sigma, H} + \|[\ubar_{tt}, \vbar_t]\|^2_{0,\sigma, H} \right ) \, d \tau.
 		 \label{eq:thm17zzz}
 		\end{align}
 	Using this in \eqref{eq:thm17xxx} we obtain (for each $\xi_i$)
 	\begin{align}
 		  \|[\abar , \bbar]  \|^2_{0, \sigma, \real^3 \times [0,T]} &
 		 + \|D (\abar + \theta \cdot \bbar) \|^2_{0, \sigma, \real^3 \times [0,T]}
 		  \cleq \, \frac{1}{\sigma} \| [\abar, \bbar, D \abar, D \bbar]\|^{{2}}_{0,\sigma, \real^3 \times [0,T]}
 		  \nn
 		  \\
  & + \int_{\tau_{min}}^{\tau_{max}} 
  \sum_{i=1}^4 \left ( \|\ubar\|^2_{2,\sigma,H} 
 		 + \| [\ubar_t, \vbar]\|^2_{1, \sigma, H} + \|[\ubar_{tt}, \vbar_t]\|^2_{0,\sigma, H} \right ) \, d \tau.
 		 \label{eq:thm17aaa}
 		\end{align}
 	Now 
 	\[
 	|D \abar + \theta \cdot D \bbar| \cleq | [\abar, \bbar] | + |D (\abar + \theta \cdot \bbar)| 
 	\]
 	and the locations $\xi_1, \cdots, \xi_4$ are diverse with respect to $(T+1)B$, hence
 	\begin{align*}
 		  \|[\abar , \bbar, D \abar, D \bbar]  \|^2_{0, \sigma, \real^3 \times [0,T]} 
 		\, & \cleq \, \frac{1}{\sigma} \| [\abar, \bbar, D \abar, D \bbar]\|^{{2}}_{0,\sigma, 
 		\real^3 \times [0,T]}
 		 \nn
 		 \\
 		 & \qquad + \int_{\tau_{min}}^{\tau_{max}}  \sum_{i=1}^4 \left ( \|\ubar\|^2_{2,\sigma,H} 
 		 + \| [\ubar_t, \vbar]\|^2_{1, \sigma, H} + \|[\ubar_{tt}, \vbar_t]\|^2_{0,\sigma, H} \right ) 
 		 \, \d \tau.
 		\end{align*}
 		so taking $\sigma$ large enough, we obtain
 		\[
 		 \|[\abar , \bbar, D \abar, D \bbar]  \|^2_{0, \sigma, \real^3 \times [0,T]} 
 		\,  \cleq \,
 		  \int_{\tau_{min}}^{\tau_{max}}  \sum_{i=1}^4 \left ( \|\ubar\|^2_{2,\sigma,H} 
 		 + \| [\ubar_t, \vbar]\|^2_{1, \sigma, H} + \|[\ubar_{tt}, \vbar_t]\|^2_{0,\sigma, H} \right ) 
 		 \, \d \tau.
 		 \]
 	Now $\cbar = \abar_t - \nabla \cdot \bbar$, so the proof of Theorem \ref{thm:abcstab} is complete.
	
	\vspace{0.4in}
 		

{
\begin{lemma}\label{decomposition}
	For $\alpha$ defined in \eqref{eq:alphadef}, we have 
	\begin{align*}
		\mc{L} \alpha = \alpha \left(\dd_t - \theta \cdot \nabla \right) (a + \theta \cdot b) - 
		\Delta_S \alpha + 2 b^\perp \cdot \nabla \alpha 
		- ( |b^\perp|^2 + 2 r^{-1} \theta \cdot b) \alpha,
	\end{align*}
    where $\Delta_S$ denotes the spherical Laplacian and $b^\perp = b - (b\cdot\theta)\theta$. 
\end{lemma}
\begin{proof}
	For convenience, we consider $\xi=0$. The general case follows by translation. 
	We have
	\begin{align}
		\nn \mc{L}\alpha & = \left(\dd_t^2 - \dd_r^2 - \f{2}{r} \dd_r - \Delta_S - 2 a \dd_t + 2 b \cdot \nabla + a^2 - |b|^2\right) \alpha \\
		\nn  & = \f{1}{r} \left(r\dd_t^2 - r\dd_r^2 - 2\dd_r \right) \alpha - \left(\Delta_S + 2 a \dd_t - 2 b \cdot \nabla - a^2 + |b|^2\right) \alpha \\
		\nn  & = \f{1}{r} \left(\dd_t^2 - \dd_r^2\right) (r\alpha) - \left(\Delta_S + 2 a \dd_t - 2 b \cdot \nabla - a^2 + |b|^2\right) \alpha \\
		\label{using trans eqn in decomp 1}  & = \f{1}{r} \left(\dd_t - \dd_r\right)\left(\left(\dd_t + \dd_r\right) r\alpha \right) - \left(\Delta_S + 2 a \dd_t - 2 b \cdot \nabla - a^2 + |b|^2\right) \alpha \\
		\nn  & = \f{1}{r} \left(\dd_t - \dd_r\right) \left( r \alpha (a + \theta \cdot b)\right) - \left(\Delta_S + 2 a \dd_t - 2 b \cdot \nabla - a^2 + |b|^2\right) \alpha \\
		\nn  & = \alpha (\dd_t - \dd_r) (a + \theta \cdot b) - \f{(a + \theta \cdot b) \alpha}{r} + (a + \theta \cdot b) (\alpha_t - \alpha_r) 
		\\
		\nn
		& \qquad - \left(\Delta_S + 2 a \dd_t - 2 b \cdot \nabla - a^2 + |b|^2\right) \alpha;
	\end{align}
	we used \eqref{transport eqn} in \eqref{using trans eqn in decomp 1}. As a consequence, we have
	\begin{align}
		\mc{L}\alpha  &- \alpha  (\dd_t - \dd_r) (a + \theta \cdot b) + \Delta_S \alpha 
		\nn \\
		& =  - \f{(a + \theta \cdot b) \alpha}{r} + (a + \theta\cdot b) (\alpha_t - \alpha_r)  
		-2 a \alpha_t + 2( (\theta \cdot b) \theta + b^\perp) \cdot \nabla \alpha 
		+ \left( a^2 - |b|^2\right) \alpha 
		\nn \\
		& =  - \f{(a + \theta \cdot b) \alpha}{r} + (a + \theta\cdot b) (\alpha_t - \alpha_r)  
		-2 a \alpha_t + 2 (\theta \cdot b) \alpha_r + 2b^\perp \cdot \nabla \alpha 
		+ \left( a^2 - |b|^2\right) \alpha 
		\nn \\
		& = - \f{(a + \theta \cdot b) \alpha}{r} - (a - \theta\cdot b) (\alpha_t + \alpha_r) 
		+ 2 b^\perp \cdot \nabla \alpha + \left( a^2 - |b|^2\right) \alpha 
		\label{eq:thm17ppp}\\
		& = - \f{(a + \theta \cdot b) \alpha}{r} - (a - \theta\cdot b) (a + \theta \cdot b 
		- r^{-1}) \alpha + 2 b^\perp \cdot \nabla \alpha + \left( a^2 - |b|^2\right) \alpha 
		\nn \\
		& = 2 b^\perp \cdot \nabla \alpha - 2 r^{-1} (\theta \cdot b) \alpha 
		- \left(a^2 - (\theta \cdot b)^2\right) \alpha + \left( a^2- |b|^2\right) \alpha 
		\nn \\
		& = 2 b^\perp \cdot \nabla \alpha - 2 r^{-1} (\theta \cdot b) \alpha 
		- |b^{\perp}|^2 \alpha,
		\nn
	\end{align}
	where we used \eqref{transport eqn} again in \eqref{eq:thm17ppp}.
\end{proof}	
}
 		

\section{The forward problems}\label{Wellposedness}

We give proofs of Propositions \ref{Heaviside} and \ref{Delta} using the standard progressing wave expansion method; 
one has to go through the computations to be certain that everything works, particularly in a cylindrical neighborhood
of $\{\xi\} \times \real$. 
It is enough to prove the proposition when $\tau=0$, $\xi=0$ since the general case follows from a translation argument. So $a,b,q$ are compactly supported smooth functions on $\real^3 \times \real$ which are zero on
$B_\ep \times \real$ for some $\ep>0$. Here $B_\ep$ is the origin centered open ball of radius $\ep$. Also, we write
$U(x,t; 0,0), V(x,t;0,0), \alpha(x,t;0)$ as $U(x,t), V(x,t), \alpha(x,t)$.

The uniqueness of the distributional solution follows from the proof of uniqueness for Proposition \ref{prop:infivp}. It remains to prove the existence and the structure of $U, V$.

We recall
\[
\mc{M} = -2a\dd_t + 2b\cdot\nabla + q , \qd \Tcal = \dd_t + \theta\cdot\nabla - (a+\theta\cdot b) + r^{-1}
\]
and 
\begin{align*}
\Lcal := (\pa_t -a)^2 - (\nabla -b)^2 + c 
= \Box - 2a\dd_t + 2b\cdot\nabla + q
= \Box + \Mcal
\end{align*}
Also, for $x \neq 0$ we define $r = |x|$ and $\theta = x/|x|$. We will need two observations, described next, in the proofs of Propositions \ref{Heaviside} and \ref{Delta}. 

For an arbitrary smooth function $h$ on $(\real^3 \setminus \{0\}) \times \real$ and an arbitrary distribution $F$ on $\real$, we claim
\begin{align}\label{gen expansion}
 \mc{L} \left(h(x,t) F(t - |x|)\right) = 2 \Tcal (h) F' (t- |x|) + \left(\mc{L}h \right) F(t - |x|),
 \qquad x \neq 0.
 \end{align}
 We give its brief derivation. For $x \neq 0$, we have
 		\begin{align*}
 		   (\dd_t - a)\left(h \, F(t - |x|) \right) & = h F'(t - |x|) 
 		   + F(t - |x|)(\dd_t - a)h, \\
 		   (\dd_t - a)^2\left(h \, F(t - |x|) \right) & = h F''(t - |x|) 
 		   + 2 F'(t - |x|)(\dd_t - a)h + F(t - |x|)(\dd_t - a)^2 h.
 		\end{align*}
  Similarly
 		 \begin{align*}
 		     (\nabla - b)& ( h(x, t) F(t - |x|)) 
 		      = - \theta h F'(t - |x|) + F(t -|x|) (\nabla - b) h, \\
 		     (\nabla - b)^2 & ( h(x, t) F(t - |x|) ) 
 		     \\
 		     & = - (\nabla - b) \cdot \left( \theta h F'(t - |x|) \right)
 		      + (\nabla - b) \cdot \left( F(t - |x|) (\nabla - b) h \right) \\
 		     & = h F''(t - |x|) -  F'(t - |x|) (\nabla - b)\cdot  
 		     \left(\theta h\right)  \\
 		     & \ \   - \theta \cdot \left((\nabla -b) h\right) F'(t - |x|) 
 		     + \left( \left(\nabla - b\right)^2 h \right) F(t - |x|)\\
 		     & = h F''(t - |x|)  - 2 \left(\theta \cdot (\nabla - b) h + r^{-1} h  \right) F'(t - |x|)
 		  + \left( \left(\nabla - b\right)^2 h \right) F(t - |x|)
 		 \end{align*} 
 		 Hence (\ref{gen expansion}) follows.

We will also need the solution of the transport equation
\beqn
(\Tcal f)(r \theta, t_0 + r) = g(r \theta, t_0 + r), \qquad r \neq 0.
\label{eq:Tfg}
\eeqn
We summarise the claim as the following lemma.
\begin{lemma}\label{lemma:Tfg}
Suppose $g(x,t)$ is a smooth function on $\real^3 \times \real$ which is zero on $B_\ep \times \real$ and the 
restriction of $g$ to the region $t \leq T$ is compactly supported for each $T$. Then \eqref{eq:Tfg} has a solution
given by
\beqn
f(r \theta, r + t_0) = \alpha(r \theta, t_0 + r) \int_0^r \frac{g(s \theta, t_0 + s)}{ \alpha(s\theta, t_0+s)} \, ds
\label{eq:Tfgsol}
\eeqn
with $f$ smooth on $\real^3 \times \real$ and zero on $B_\ep \times \real$. Further, the restriction of $f$ to
$t \leq T$ is compactly supported with 
\beqn
\| f \|_{C^p( \real^3 \times (-\infty, T])}
\leq C \| g \|_{C^p( \real^3 \times (-\infty, T])},
\label{eq:fgpest}
\eeqn
with the constant $C$ determined by $\ep$, $T$ and $\|[a,b]\|_{C^p(\real^3 \times (-\infty,T])}$.
\end{lemma}
We give the short proof of the lemma. For any smooth function $f(x,t)$ and any $t_0$ we have
\[
\frac{d}{dr} ( r f(r \theta, t_0 + r) ) = [ r ( f_t + \theta \cdot \nabla f) + f](r \theta, t_0 + r),
\qquad r \neq 0,
\]
hence
\beqn
r (\Tcal f)(r \theta, t_0 + r) = \frac{d}{dr} (r f(r \theta, t_0 + r) ) 
- (a + \theta \cdot b)( r f(r \theta, t_0 + r) ),
\qquad r \neq 0.
\label{eq:rTf}
\eeqn
Therefore \eqref{eq:rTf} may be rewritten as the ODE
\[
\frac{d}{dr} (r f(r \theta, t_0 + r) ) 
- [(a + \theta \cdot b)( r f)](r \theta, t_0 + r) ) = r g(r \theta, t_0 + r), \qquad r \neq 0.
\]
An integrating factor for this ODE is (note $a,b$ are zero in $B_\ep \times \real$)
\begin{align*}
\exp \left ( - \int_0^r (a+ \theta \cdot b)(s \theta, t_0+s) \, ds \right )
& = \exp \left ( - \int_0^r (a+ \theta \cdot b)((r-s) \theta, t_0+r-s) \, ds \right )
\\
& = \frac{1}{r \alpha(r \theta, t_0+r)}
\end{align*}
so the ODE may be rewritten as
\[
\frac{d}{dr} \left ( \frac{f(r \theta, t_0 + r)}{ \alpha(r \theta, t_0+r)} \right ) = 
\frac{g(r \theta, t_0 + r)}{ \alpha(r \theta, t_0+r)},
\qquad r \neq 0.
\]
Hence, one solution of \eqref{eq:Tfg} is
\[
f(r \theta, t_0 + r) = \alpha(r \theta, t_0 + r) \int_0^r \frac{g(s \theta, t_0 + s)}{ \alpha(s\theta, t_0+s)} \, ds, 
\qquad r \neq 0;
\]
note that $f(x,t)$, is zero in $B_\ep \times \real$ and smooth on $\real^3 \times \real$. Further
$f$ is compactly supported when restricted to $t \leq T$ and
\[
\| f \|_{C^p( \real^3 \times (-\infty, T])}
\leq C \| g \|_{C^p( \real^3 \times (-\infty, T])},
\]
with the constant $C$ determined by $\ep$, $T$ and $\|[a,b]\|_{C^p(\real^3 \times (-\infty,T])}$.

\subsection{Proof of Proposition \ref{Heaviside}}
\label{subsec:Heaviside}
    
   One may verify with a standard argument that
 		\begin{align}\label{fund soln}
 		    \square \left( \f{H(t - |x|)}{|x|} \right) = 4 \pi H(t) \delta(x).
 		\end{align}
 	We seek a solution $U(x,t)$, of the IVP (\ref{IVP for Heaviside 1}), (\ref{IVP for Heaviside 2}), of the form
 		\begin{align*}
 		    U(x, t) = \f{H(t - |x |)}{|x|} + u(x, t) H(t - |x|),
 		\end{align*}
 		with $u(x,t)$ a smooth function in the region $t \geq |x|$, satisfying
 		$u(x,t)=0$ in a neighborhood of $(x=0, t=0)$. 
 		
 		Clearly such a $U$ satisfies (\ref{IVP for Heaviside 2}). So we need to find such a $u(x,t)$ so that
 \[
 \Lcal (u(x,t) H(t -|x|) ) = - \Mcal \left ( \frac{H(t-|x|)}{|x|} \right ).
 \label{eq:LuM}
 \]
 %
Since $\Mcal=0$, $a=0, b=0$ on $B_\ep \times \real$, we have
\begin{align*}
\Mcal \left ( \frac{ H(t-|x|)}{|x|} \right )
& = \Mcal ( |x|^{-1} ) H(t-|x|) - 2 |x|^{-1} \, (a + \theta \cdot b) \, \delta(t-|x|),
\end{align*}
hence we want
\beqn
\Lcal [ u(x,t) H(t-|x|) ] =  2 |x|^{-1} (a+  \theta \cdot b) \,  \delta( t - |x|)
- \Mcal( |x|^{-1}) \, H(t-|x|).
\label{eq:Utemp1}
\eeqn

For $N$ large enough to be chosen later, we seek $u(x,t)H(t-|x|)$ in the form
 	\begin{align}\label{solution:u}
 	    u(x,t)H(t-|x|) =  a_0(x,t) H(t-|x|) + \sum_{k=1}^N a_k(x,t) \frac{(t - |x|)_+^k}{k!} + S_N(x,t)
 	\end{align}
 	for suitably chosen smooth functions $a_k$ which will be zero on $B_\ep \times \real$, and a $S_N$ which will be 
 	highly differentiable as $N$ increases, zero in a neighborhood of $(0,0)$ and supported in $t \geq |x|$. With such an expansion we will define
 	\beqn
 	u(x,t) := \sum_{k=0}^N a_k(x,t) \frac{(t - |x|)_+^k}{k!} + S_N(x,t).
 	\label{eq:udefSN}
 	\eeqn
 	
We construct the $a_k$ and $S_N$ so that \eqref{eq:Utemp1} holds. Since the $a_k$ are to be zero on $B_\ep \times \real$,
it is clear that
\[
\Lcal \left ( a_0(x,t) H(t-|x|) + \sum_{k=1}^N a_k(x,t) \frac{(t - |x|)_+^k}{k!} \right ) = 0,
\qquad \text{on } B_\ep \times \real.
\]
Further, for $x \neq 0$, noting that
 	\[
H(s) = s_+^0, \qquad \frac{d}{ds} \delta(s) = \delta'(s),
 \qquad \frac{d}{ds} H(s) = \delta(s),
 \qquad \frac{d}{ds} \frac{s_+^k}{k!} = \frac{s_+^{k-1}}{(k-1)!}, \qquad k \geq 1,
 	\]
and	using \eqref{gen expansion} we obtain (for $x \neq 0$)
 \begin{align*}
 \Lcal & \left ( a_0(x,t) H(t-|x|) + \sum_{k=1}^N a_k(x,t) \frac{(t - |x|)_+^k}{k!} \right )
 \\
&= 2 (\Tcal a_0) \delta(t-|x|) + (\Lcal a_0) H(t-|x|)
+ \sum_{k=1}^N 2 (\Tcal a_k) \frac{(t-|x|)_+^{k-1}}{(k-1)!} 
 	    + (\Lcal a_k) \frac{(t-|x|)_+^k}{k!} 
 	    \\
& = 2 (\Tcal a_0) \delta(t-|x|) 
+ \sum_{k=1}^N (2\mc{T} a_k + \Lcal a_{k-1}) \frac{(t - |x|)_+^{k-1}}{(k-1)!} + 
\Lcal a_N \frac{(t - |x|)_+^N}{N!}.
\end{align*}
 Keeping in mind \eqref{eq:Utemp1}, we choose $a_0(x,t)$ and $a_1(x,t)$ so that
 \begin{align}
 2 \Tcal a_0 =2 |x|^{-1} (a+ \theta \cdot b),  & \qquad \text{on } x \neq 0
 \label{eq:a0de}
 \\
 2 \Tcal a_1 + \Lcal a_0 = - \Mcal (|x|^{-1}), & \qquad \text{on } ~ x \neq 0,
 \label{eq:a1de}
 \end{align}
 and, for $2 \leq k \leq N$, we choose $a_k$ so that
 \begin{align}
 2 \mc{T} a_k + \mc{L} (a_{k-1}) = 0, \qquad \text{on } ~ x \neq 0.
 \label{eq:akde}
 \end{align}

 Assuming for the moment that we have constructed smooth $a_k$ satisfying these equations with
 $a_k$ zero on $B_\ep \times \real$, keeping in mind \eqref{eq:Utemp1}, we need to find $S_N$ which solves
 \begin{align}
\Lcal S_N = - \left(\Lcal a_N \right)\, \frac{(t-|x|)_+^N}{N!}, & \qquad \text{on } \real^3 \times \real,
\label{eq:SNde}
\\
S_N = 0, & \qquad \text{on } t<0.
\label{eq:SNic}
 \end{align}
Since $a_N$ is a smooth function that is zero in $B_\ep \times \real$, 
the function $\Lcal a_N \frac{(t-|x|)_+^N}{N!}$ is in $C^{N-1}(\real^3 \times \real)$, zero in a neighborhood of
$(0,0)$ and supported in the region $t \geq |x|$. Hence, if $N>5$, then by Proposition \ref{prop:infivp} with 
$m=N-1$ the IVP \eqref{eq:SNde}, \eqref{eq:SNic} has a unique distributional solution which is in $C^{N-3}(\real^3 \times \real)$. Further
\beqn
\| S_N \|_{C^{N-3}(\real^3 \times (-\infty, T])} \leq C \| \Lcal a_N \|_{C^{N-1}(\real^3 \times (-\infty,T])}.
\label{eq:SNest},
\eeqn
with $C$ determined by $T$ and $\|[a,b,q]\|_{C^{N-1}(\real^3 \times (-\infty,T])}$.

Hence $S_N$ is at least $C^2$, so by a 
standard energy estimate argument, one can show that $S_N$ is supported in the region $t \geq |x|$ and $S_N=0$ in a neighborhood of $(0,0)$. Hence, if we take $N>5$ then the $u$ defined by \eqref{eq:udefSN}
is in $C^{N-3}(\real^3 \times \real)$, zero in a neighborhood of $(0,0)$ and $u(x,t) H(t-|x|)$ is
the (unique) distributional solution of the IVP \eqref{IVP for Heaviside 1A}, \eqref{IVP for Heaviside 2B}. Now
$N$ was arbitrary and $u$ is uniquely determined on $t \geq |x|$, hence $u$ is smooth on $t \geq |x|$. 
Since \eqref{eq:Utemp1} holds, we see that \eqref{ch BVP for Heaviside 1} holds.

If remains to prove that there are smooth $a_k(x,t)$ which satisfy \eqref{eq:a0de}, \eqref{eq:a1de}, \eqref{eq:akde}, 
are zero in $B_\ep \times \real$, that for the $u$ defined by \eqref{eq:udefSN} 
the relation \eqref{ch BVP for Heaviside 2} holds, and we have the estimate on $\|u\|_{C^3(Q_{0,0})}$ claimed in 
Proposition \ref{Heaviside}.
 	
Since the RHS \eqref{eq:a0de} is smooth, compactly supported and zero on 
$B_\ep \times \real$, from Lemma \ref{lemma:Tfg}, we can construct a smooth $a_0$ satisfying 
\eqref{eq:a0de}, which is  zero on $B_\ep \times \real$ and its restriction to $t \leq T$ is compactly supported. Further
\[
\| a_0\|_{C^p( \real^3 \times (-\infty, T])}
\leq C
\]
with $C$ dependent only on $\ep$, $T$ and $\|[a,b]\|_{C^p(\real^3 \times (-\infty,T])}$.
Next, rewriting \eqref{eq:a1de} as
\[
\Tcal a_1 = - \frac{1}{2} ( \Lcal a_0 + \Mcal (|x|^{-1}) ), \qquad x \neq 0,
\]
again,  from Lemma \ref{lemma:Tfg}, there is a smooth solution $a_1$ of \eqref{eq:a1de}, which is zero on 
$B_\ep \times \real$ and its restriction to $t \leq T$ is compactly supported. Further, using the estimate on $a_0$
\begin{align*}
\| a_1 \|_{C^p( \real^3 \times (-\infty, T])}
& \leq C \| \Lcal a_0 + \Mcal (|x|^{-1}) \|_{C^p( \real^3 \times (-\infty, T])}
\\
& \leq C ( \|a_0\|_{C^{p+2}(\real^3 \times (-\infty, T])} + \|[a,b,q]\|_{C^p( \real^3 \times (-\infty, T])} )
\\
& \leq C_1
\end{align*}
where $C_1$ is a constant determined by $\ep$, $T$ and $\|[a,b,q]\|_{C^{p+2}( \real^3 \times (-\infty, T])}$. Next,
for any $2 \leq k \leq N$ we may write \eqref{eq:akde} as
\[
\Tcal a_k = - \frac{1}{2} \Lcal a_{k-1}, \qquad x \neq 0,
\]
hence, from Lemma \ref{lemma:Tfg}, there is a smooth solution $a_k$ of \eqref{eq:akde} which is zero on 
$B_\ep \times \real$ and its restriction to $t \leq T$ is compactly supported. Further
\[
\| a_k \|_{C^p( \real^3 \times (-\infty, T])}
\leq C \| \Lcal a_{k-1} \|_{C^p( \real^3 \times (-\infty, T])}
\leq C_2 \| a_{k-1} \|_{C^{p+2}( \real^3 \times (-\infty, T])}.
\]
with $C_2$ determined by $\ep, T$ and $\|[a,b,q]\|_{C^p( \real^3 \times (-\infty, T])}$. So by induction,
\[
\| a_k \|_{C^p( \real^3 \times (-\infty, T])} \leq C
\]
with $C$ determined by $\ep,T,k$ and $\|[a,b,q]\|_{C^{p+2k}( \real^3 \times (-\infty, T])}$ for $0 \leq k \leq N$.
In particular 
\[
\| a_N \|_{C^p( \real^3 \times (-\infty, T])} \leq C
\]
with $C$ determined by $\ep, T, N$ and $\|[a,b,q]\|_{C^{p+2N}( \real^3 \times (-\infty, T])}$.
We use this estimate in \eqref{eq:SNest} to obtain
\[
\|S_N\|_{C^{N-3}(\real^3 \times (-\infty,T])} \leq C \| a_N \|_{C^{N+1}(\real^3 \times (-\infty,T])}
\leq C_1
\]
with $C_1$ determined by $\ep, T$ and $\|[a,b,q]\|_{C^{3N+1}(\real^3 \times (-\infty,T]}$. In particular, taking
$N=6$ we have
\[
\|S_6\|_{C^3(\real^3 \times (-\infty,T])} \leq C
\]
where $C$ is determined by $\ep, T$ and $\|[a,b,q]\|_{C^{19}(\real^3 \times (-\infty,T]}$. 

From the uniqueness of the distributional solution, we know that the $u$ defined by \eqref{eq:udefSN} is independent of 
$N$ on the region $t \geq |x|$, hence using the estimates on $a_0, a_1, \cdots, a_6$ and $S_6$ we have
\[
\|u\|_{C^3(Q_{0,0})} \leq C
\]
with $C$ determined by $\ep,T$ and $\|[a,b,q]\|_{C^{19}(\real^3 \times (-\infty,T]}$. Here $Q_{0.0}$ is
$\{ (x,t): |x| \leq t \leq T\}$.

Finally, since $S_6$ is in $C^3$ and supported on $t \geq |x|$, we observe from \eqref{eq:udefSN} that
\[
u(x,|x|) = a_0(x,|x|), \qquad x \in \real^3.
\]
From Lemma \ref{lemma:Tfg} applied to \eqref{eq:a0de}, and taking $t_0=0$, we have
\begin{align*}
a_0(r \theta, r) & = \alpha(r \theta,r) \int_0^r \frac{(a + \theta \cdot b)(s \theta, s)}{s \alpha(s \theta, s)} \, ds.
\end{align*}
Now, from \eqref{eq:alphadef}
\begin{align*}
\frac{(a + \theta \cdot b)(s \theta, s)}{s \alpha(s \theta, s)} & =
(a + \theta \cdot b)(s \theta, s) \, \exp \left ( - \int_0^s (a+ \theta \cdot b)((s-\rho) \theta, s-\rho) \, 
d \rho \right )
\\
& = (a + \theta \cdot b)(s \theta, s) \, \exp \left ( - \int_0^s (a+ \theta \cdot b)(\rho \theta,\rho) \, 
d \rho \right )
\\
& = - \frac{d}{ds} \left [ \exp \left ( - \int_0^s (a+ \theta \cdot b)(\rho \theta, \rho) \, d \rho \right ) \right ]
\\
& = - \frac{d}{ds} \left ( \frac{1}{s\alpha(s \theta, s)} \right ).
\end{align*}
Noting that $ \lim_{r \to 0^+} ( r \alpha(r \theta, r) )=1$, we obtain
\begin{align*}
a_0(r \theta, r) = \alpha(r \theta, r) - r^{-1},
\end{align*}
proving \eqref{ch BVP for Heaviside 2}.

 		  
   \subsection{Proof of Proposition \ref{Delta}} 
  We have $\xi=0, \tau=0$ and $a,b,c$ are zero in $B_\ep \times \real$. We seek a 
   solution of the IVP \eqref{IVP for delta 1A} - \eqref{IVP for delta 2B} in the form 
\begin{align}
V(x, t) = |x|^{-1} \delta(t-|x|) + f(x, t) \delta (t  - |x|) + v(x, t) H(t - |x|),
\label{eq:Vfv}
\end{align}
with $v(x,t)$ a smooth function in $t \geq |x|$, $f(x,t)$ a smooth function on $\real^3 \times \real$,
$v(x,t)$ zero in a neighborhood of ($x \smeq 0,t \smeq 0$) and $f(x,t)$ zero on $B_\ep \times \real$.
Clearly such a $V(x,t)$ will satisfy the initial condition (\ref{IVP for delta 2B}), so we just need to 
find a solution of this form for \eqref{IVP for delta 1A}.

Since
\[
\Box (  |x|^{-1} \delta(t-|x|) ) = 4 \pi \delta(x) \delta(t)
\]
we have (note $\Mcal, a, b$ are zero in $B_\ep \times \real$)
\begin{align*}
\Lcal ( |x|^{-1} \delta(t-|x|) ) & = \Mcal ( |x|^{-1} \delta(t-|x|) )
= \Mcal(|x|^{-1}) \delta(t-|x|) - 2 |x|^{-1} (a+ \theta \cdot b) \, \delta'(t-|x|).
\end{align*}
Hence using \eqref{gen expansion} (we assume $f=0$ in $B_\ep \times \real$) we have
\begin{align*}
\Lcal & ( |x|^{-1} \delta(t-|x|) + f(x,t) \delta(t-|x|) ) - 4 \pi \delta(x) \delta(t)
\\
& \qquad =  2 [ \Tcal f -  |x|^{-1} (a+ \theta \cdot b)]  \, \delta'(t-|x|) 
+ [\Lcal f + \Mcal(|x|^{-1} )] \delta(t-|x|).
\end{align*}
Since $|x|^{-1} (a+ \theta \cdot b)$ is zero on $B_\ep \times \real$, from Lemma \ref{lemma:Tfg}, we can find a smooth
$f(x,t)$ which is zero on $B_\ep \times \real$ and $\Tcal f = |x|^{-1} (a+ \theta \cdot b)$. In fact, from Lemma \ref{lemma:Tfg}, we have
\begin{align*}
f(r \theta, r + t_0) & = \alpha(r \theta, t_0+r) \int_0^r \frac{(a + \theta \cdot b)(s \theta, t_0+s)}{s \alpha(s \theta, t_0+s)} \, ds
\\
& = \alpha(r \theta, t_0+r) - r^{-1},
\end{align*}
by the calculation at the end of subsection \ref{subsec:Heaviside}. Hence
\[
f(x,t) = \alpha(x,t) - |x|^{-1}.
\]
Note that, from \eqref{eq:alphadef}, we have
$\alpha(x,t) - |x|^{-1}=0$ in $B_\ep \times \real$.

So, keeping in mind \eqref{eq:Vfv} and \eqref{IVP for delta 1A}, we seek $v(x,t)$, a smooth function on $t \geq |x|$ which is zero near $(0,0)$ and
\beqn
\Lcal [ v(x,t) H(t-|x|) ]
= [ \Lcal (\alpha(x,t) - |x|^{-1}) + \Mcal (|x|^{-1} )] \, \delta(t-|x|).
\label{eq:Lva}
\eeqn
We seek $v(x,t) H(t-|x|)$ in the form
\beqn
v(x,t) H(t-|x|) = \sum_{k=0}^N b_k(x,t) \frac{ (t-|x|)_+^k}{k!} + R_N(x,t),
\label{eq:vaSN}
\eeqn
for some large $N$, for smooth functions $b_k$ which are zero in $B_\ep \times \real$ and for some regular enough 
function $R_N$ which is supported in $t \geq |x|$ and zero in a neighborhood of $(0,0)$. Then we will take
\beqn
v(x,t) = \sum_{k=0}^N b_k(x,t) \frac{ (t-|x|)^k}{k!} + R_N(x,t), \qquad t \geq |x|.
\label{eq:vdef}
\eeqn

As seen in the proof of Proposition \ref{Heaviside}, we have
\begin{align*}
 \Lcal & \left ( \sum_{k=0}^N b_k(x,t) \frac{(t - |x|)_+^k}{k!} \right )
 \\
& = 2 (\Tcal b_0) \delta(t-|x|) 
+ \sum_{k=1}^N (2\mc{T} b_k + \Lcal b_{k-1}) \frac{(t - |x|)_+^{k-1}}{(k-1)!} + 
(\Lcal b_N) \frac{(t - |x|)_+^N}{N!}.
\end{align*}
So keeping in mind \eqref{eq:Lva}, we seek $b_k$ such that
\begin{align}
\Tcal b_0 & = \Lcal (\alpha(x,t) - |x|^{-1}) + \Mcal (|x|^{-1} ),
\qquad x \neq 0,
\label{eq:b0de}
\\
\Tcal b_k & = - \frac{1}{2} \Lcal b_{k-1}, \qquad x \neq 0, ~~ 1 \leq k \leq N.
\label{eq:bkde}
\end{align}
Since the RHS of \eqref{eq:b0de} is smooth, zero on $B_\ep \times \real$ and its restriction to $t \leq T$ is compactly 
supported, Lemma \ref{lemma:Tfg} guarantees a smooth solution $b_0$ of \eqref{eq:b0de} with $b_0$ zero on 
$B_\ep \times \real$ and its restriction to $t \leq T$ compactly supported. Further
\[
\|b_0\|_{C^p(\real^3 \times (-\infty,T])} \leq C
\]
with $C$ determined by $\ep, T$ and $\|[a,b,q]\|_{C^{p+2}(\real^3 \times (-\infty,T])}$.
Applying Lemma \ref{lemma:Tfg} recursively
to \eqref{eq:bkde} we conclude that, for $1 \leq k \leq N$, there is a smooth solution $b_k$ of \eqref{eq:bkde}
with $b_k$ zero on $B_\ep \times \real$ and its restriction to $t \leq T$ compactly supported. Further
\[
\|b_k\|_{C^p(\real^3 \times (-\infty,T])} \leq C \|b_{k-1}\|_{C^{p+2}(\real^3 \times (-\infty,T])},
\qquad 1 \leq k \leq N,
\]
with $C$ determined by $\ep, T$ and $\|[a,b]\|_{C^{p+2}(\real^3 \times (-\infty,T])}$. Hence, by an induction argument,
\[
\|b_N\|_{C^p(\real^3 \times (-\infty,T])} \leq C
\]
with $C$ determined by $\ep, T$ and $\|[a,b,q]\|_{C^{p+2N+2}(\real^3 \times (-\infty,T])}$.

With the $b_k$ chosen above we have
\begin{align*}
\Lcal \left ( \sum_{k=0}^N b_k(x,t) \frac{ (t-|x|)_+^k}{k!} + R_N(x,t) \right )
= (\Lcal b_N) \frac{(t - |x|)_+^N}{N!} + \Lcal R_N.
\end{align*}
Hence for \eqref{eq:Lva} to hold, we need to find a $R_N$ which is supported on $t \geq |x|$, zero in a neighborhood 
of $(0,0)$ and is the solution of the IVP
\begin{align*}
\Lcal R_N = - (\Lcal b_N) \frac{(t - |x|)_+^N}{N!}, & \qquad \text{on } \real^3 \times \real,
\\
R_N = 0, & \qquad \text{for } t <<0.
\end{align*}
Then repeating the argument used in the proof of Proposition \ref{Heaviside}, we can show that the $v$ defined by
\eqref{eq:vdef} is smooth on $t \geq |x|$, zero in a neighborhood of $(0,0)$ and
\[
\|v\|_{C^3(Q_{0,0})} \leq C
\]
with $C$ determined by $\ep, T$ and $\|[a,b,q]\|_{C^{21}(\real^3 \times (-\infty,T])}$.

Noting that \eqref{eq:Lva} implies \eqref{ch BVP for delta 1}, it remains to verify (\ref{ch BVP for delta 2}). From
\eqref{eq:vdef}, we see that for $ t \geq |x|$ we have 
\[
v(x,t) = b_0(x,t) + b_1(x,t) (t-|x|) + \sum_{k=2}^6 b_k(x,t) (t-|x|)^k/k! + R_6(x,t).
\]
Since $R_6$ is supported in $t \geq |x|$ and is at least $C^2$, we see that $\Tcal R_6 =0$ on $t=|x|$. Further, on 
$t=|x|$
\begin{align*}
\Tcal ( b_1(x,t) (t-|x|) ) & = (\Tcal b_1(x,t)) (t-|x|) + b_1 \Tcal (t-|x|)
= b_1(x,t) (\pa_t + \theta \cdot \nabla )(t-|x|) = 0.
\end{align*}
Hence, noting that $\Lcal = \Box + \Mcal$ and that $\Box (|x|^{-1}) =0$ for $x \neq 0$, on $t=|x|$ we have
\[
(\Tcal v)(x,t) = (\Tcal b_0)(x,t) = \Lcal ( \alpha(x,t) - |x|^{-1} ) + \Mcal (|x|^{-1} )
= \Lcal \alpha)(x, t), \qquad x \neq 0.
\]


\section{Proof of Proposition \ref{prop:carleman}}\label{sec:carleman}

It is sufficient to prove the proposition when $\mc{L} = \Box$ since the lower order terms can be absorbed 
in the LHS of the inequality. This argument also shows that the constant in the inequality depends only
on $T, |\xi|, |\tau|$ and $\|[a,b]\|_{C^1(Q_{\xi,\tau})}$, $\|c\|_{C^0(Q_{\xi,\tau})}$.

We prove the proposition when $\xi=0, \tau=0$. The general $\xi, \tau$ case follows by translation. For the $\xi=0, \tau=0$ case, we denote
$Q_{\xi,\tau}, H_{\xi,\tau}$ and $C_{\xi,\tau}$ by $Q,H,C$. 

 Our proof uses Theorem A.7 of \cite{RS_1} and we keep the notation used there. We first observe that Theorem A.7 in \cite{RS_1}
 is valid for the weight $\phi(x, t) = t$. Even though it does not satisfy the strong pseudo-convexity criterion needed in Theorem A.7 of \cite{RS_1}, it does satisfy (A.25) in \cite{RS_1} which is what is needed to obtain the Carleman estimate in Theorem A.7 of \cite{RS_1}. 
 
 For our problem, $p(x, t, \xi, \tau) = -\tau^2 +\xi^2$ and $\phi(x,t) = t$. Hence
 		 \begin{align}
 		\nn & A = p(x, t, \xi, \tau) - \sigma^2 p(x, t, \nabla \phi, \phi_t ) = -\tau^2 + |\xi|^2 + \sigma^2, \\
 		\nn & B = \{p, \phi\} = p_\tau \phi_t = -2 \tau
 		\end{align}
 		implying $\{A, B\} = 0$. So (A.25) holds if we consider $g$ to be any positive constant and then choose $d>0$ accordingly. Consequently we have 
 		\begin{align}
 		   \sigma \int_Q e^{ 2\sigma t} \left( |\nabla_{x,t} w|^2 + \sigma^2 w^2  \right) + \sigma \int_{\dd Q} \nu \cdot E 
 		   \le C \int_Q e^{ 2\sigma t} |\square w|^2 
 		   \label{eq:carltemp}
 		\end{align}
 		where $\nu = (\nu_0, \nu_1, \nu_2, \nu_3)$ is  the outward unit normal and the zero index corresponds to $t$. Now $\partial Q = C \cup H$ and we compute the expressions appearing in the integral over $C$ and $H$.
 		
 	Using the calculation of the boundary terms for the wave operator from subsection A.2 in \cite{RS_1}, we write 
 		\begin{align*}
 		    \f{1}{2} E_j & = 2 (e^{\sigma t} w)_{x_j} (e^{\sigma t} w)_t - g(x,t) (e^{\sigma t} w)_{x_j} e^{\sigma t} w , \qd j\in \{1,2,3\}, \\
 		    & = e^{2 \sigma t} \left(2\sigma w w_{x_j} + 2 w_{x_j} w_t - g w w_{x_j} \right)
 		\end{align*}
 		and, 
 		\begin{align*}
 		    \f{1}{2} E_0 & = - |\nabla_{x,t} (e^{\sigma t} w) |^2 - \sigma^2 (e^{\sigma t} w)^2 +g e^{\sigma t} w (e^{\sigma t} w)_t \\
 		    & = e^{2 \sigma t} \left( -(w_t + \sigma w)^2 - |\nabla w|^2 -\sigma^2 w^2 + g w (w_t + \sigma w) \right) \\
 		    & = e^{2 \sigma t} \left( -|\nabla_{x,t} w|^2 - 2\sigma^2 w^2 - 2 \sigma w w_t + g w (w_t+ \sigma w) \right). 
 		\end{align*}

 		On $C$, we have $ \sqrt{2}\,\nu(x,t) =( -1, \theta )$, hence
 		\begin{align}
 	   \nn \nu\cdot E = & \ \rt 2 \sigma e^{2 \sigma t}
 	   \left [
 	   ( |\nabla_{x,t} w|^2 + 2 w_t \, \theta \cdot \nabla w ) + 2 \sigma w ( w_t + \theta \cdot \nabla_{x} w) 
 	   + 2\sigma^2 w^2 - \sigma g w^2 \right .\\
 	   \label{bdry term} & \hspace*{2cm}
 	   \left .- g w \left( w_t + \theta \cdot \nabla w \right) 
 	   \right ].
 	    \end{align}
 		Now
 	   \begin{align}
 	   \nn |\nabla_{x,t} w|^2+2 w_t \, \theta \cdot \nabla_{x} w = 
 	   & \  w_t^2 + 
 	   (\theta \cdot \nabla w )^2 + \sum_{i<j}\f{(\Omega_{ij}v)^2}{|x|^2}
 	   + 2 w_t \, \theta \cdot \nabla w \\
 	   \label{equality in Carleman} = & \ \sum_{i<j}\f{(\Omega_{ij} w)^2}{|x|^2}+\left( w_t+\theta \cdot\nabla w  \right)^2
 	   \end{align}
 	    where $\Omega_{ij} = x^i \partial_j - x^j \partial_i$,
 	   $i,j=1,2,3$
 	    are the angular derivatives. If we define
 	    \[
 	    P =  w_t + \theta \cdot \nabla w
 	    \]
 	    then using \eqref{equality in Carleman} in \eqref{bdry term}
 	    we obtain
 	   \begin{align}
 	   \nn \nu \cdot E & = \rt 2\sigma e^{2 \sigma t} 
 	   \left ( |x|^{-2} \sum_{i<j}( \Omega_{ij} w)^2 +
 	   P^2 + 2 \sigma P w + 2\sigma^2 w^2 - \sigma g w^2-g P w  
 	  \right )
 	   \\
 	   \nn & \ge \rt 2 \sigma e^{2 \sigma t} 
 	   \left( |x|^{-2} \sum_{i<j} (\Omega_{ij} w )^2
 	   + P^2- 2\sigma |P| |w| + 2 \sigma^2 w^2 - \sigma \|g\|_{\iy} w^2 - \|g\|_{\iy} P w \right)\\
 	   \nn & \ge \rt 2 \sigma e^{2 \sigma t} \left( |x|^{-2} \sum_{i<j}(\Omega_{ij} w )^2
 	   + P^2 \left( 1 - \epsilon - \delta \right) + w^2 \left( 2 \sigma^2 - \f{\sigma^2} {\epsilon} - \sigma \|g\|_{\iy} - \f{1} {4 \delta} \|g\|_{\iy}^2 \right) \right).
 	   \end{align}
 	   Taking $\epsilon = \f{3}{4},\ \delta = \f{1}{8}$ and $ \sigma > 0 $ large, we have 
 	    \begin{align}\label{conical bdry term 2}
 	   \nu \cdot E 
 	   & ~ \cgeq ~ \sigma e^{2 \sigma  t} 
 	   \left( |x|^{-2} \sum_{i<j} (\Omega_{ij} w )^2 + (w_t + \theta \cdot \nabla w)^2 \right) + \sigma^2 w^2,
 	   \qquad \text{on } C.
 	   \end{align}

 	   Next, on $H$, noting that $\nu=(1,0,0,0)$, using the
 	   A.M-G.M inequality as done in obtaining (\ref{conical bdry term 2}), we obtain 
 	   \begin{align}
 	   \nn \nu \cdot E = E_0 & = 2 e^{2\sigma t} \left( -|\nabla_{x,t} w|^2 - 2 \sigma^2 w^2 - 2 \sigma w w_t + g w \left( w_t + \sigma w \right)\right)\\
 	   \nn & \cgeq -  e^{2 \sigma t} \left( |\nabla_{x,t} w|^2 + \sigma^2 w^2 \right),
 	   \end{align}
 	   hence
 	   \begin{align}\label{bdry term at the top} 
 	       |\nu \cdot E| \cleq  e^{2 \sigma t} \left( |\nabla_{x,t} w|^2 + \sigma^2 w^2\right)
 	       \qquad \text{on } H.
 	   \end{align}
 	   
 	 So using \eqref{conical bdry term 2} and 
 	 \eqref{bdry term at the top} in (\ref{eq:carltemp}), we obtain
 	  \begin{align*}
 	   \sigma \int_Q e^{ 2\sigma t} \left( |\nabla_{x,t} w|^2 + \sigma^2 w^2  \right) & + \sigma \int_C { e^{ 2\sigma t}} \left( |\nabla_{C} w|^2 + \sigma^2 w^2  \right) 
 	   \\
 	  & \cleq  \int_Q e^{ 2\sigma t} |\square w|^2  + \sigma \int_H e^{ 2\sigma t} \left( |\nabla_{x,t} w|^2 + \sigma^2 w^2  \right).    
 	  \end{align*}
 	  where the constant is independent of $w,\ \sigma$ and depends only on $\tau,T$. This completes the proof of the proposition.

 
 \section{Construction of a diverse set of locations}\label{sec:diverse}
 
 Let $D$ be a non-empty bounded open subset of $\real^d$. We give two ways to construct a diverse set of locations with 
 respect $D$. If $\xi_1, \cdots, \xi_k$ is a collection of vectors in $\real^d$ then the { the affine hull of $\xi_1, \cdots, \xi_k$ 
 is}
  \[
 {\mc{A}(\xi_1, \cdots, \xi_k)} := \left \{ \sum_{i=1}^k \alpha_i \xi_i : 
 \alpha_i \in \real, ~ \sum_{i=1}^k \alpha_i = 1 \right \}
 \]
 and the convex hull of $\xi_1, \cdots, \xi_k$ is
 \[
 {\mc{C}(\xi_1, \cdots, \xi_k)} := \left \{ \sum_{i=1}^k \alpha_i \xi_i : 
 \alpha_i \geq 0, ~ \sum_{i=1}^k \alpha_i = 1 \right \}.
 \]
 
 The following proposition gives two ways to generate a diverse collection of sources.
\begin{proposition}\label{prop:diverse}
Suppose $d$ is a positive integer and $D$ is a non-empty bounded open subset of $\real^d$.
\begin{enumerate}
\item[(a)] Suppose $\xi_1, \cdots, \xi_d$ are linearly independent vectors 
in $\real^d$ and $\xi_{d+1} \in {\mc{C}(\xi_1, \cdots, \xi_d)}$ but different from $\xi_1, \cdots, \xi_d$. 
If ${\mc{A}(\xi_1, \cdots, \xi_d)}$ does not intersect $\Dbar$ then $\xi_1, \cdots, \xi_{d+1}$ is a diverse set
of locations with respect to $D$.
\item[(b)] If $\xi_1, \cdots, \xi_{d+1}$ is a set of locations in $\real^d \setminus \Dbar$ such that $\Dbar$ is in the interior of ${\mc{C}(\xi_1, \cdots, \xi_{d+1})}$ then $\xi_1, \cdots, \xi_{d+1}$ is a diverse set of locations with respect to $D$.
\end{enumerate}
\end{proposition}
\noindent
{\bf Remark.} If $\rho>0$ and $N > \rho \sqrt{d}$ then
$Ne_1, \cdots, Ne_d, N(e_1+\cdots+ e_d)/d$ is a diverse set of locations collection with respect to $\rho B$. This is so because of (a) and that
${\mc{A}(Ne_1, \cdots, Ne_d)}$ does not intersect $\rho\Bbar$.

\begin{proof}
Suppose $\xi_i \in \real^d \setminus \Dbar$, $i=1, \cdots, d+1$. 
 For any $x \in \Dbar$, define
\[
\theta_i(x) = \frac{ x - \xi_i}{|x-\xi_i|}, \qquad i=1, \cdots, d+1,
\]
and the $(d+1) \times (d+1)$ matrix
\[
M(x) = \begin{bmatrix} 1 & 1 & \cdots & 1 \\ \theta_1(x) & \theta_1(x) & \cdots & \theta_{d+1}(x) \end{bmatrix}, 
\qquad x \in \Dbar.
\]
Then $\xi_1, \cdots, \xi_{d+1}$ is a diverse set of locations with respect to $D$ iff
\[
\|[a,b]\| \cleq \| M(x)[a,b] \|,
\qquad \forall x \in \Dbar, a \in \real, b \in \real^d,
\]
with the constant independent of $x,a,b$. This condition is equivalent to the invertibility 
of $M(x)$ for all $x \in \Dbar$ because the invertibility of $M(x)$ implies the operator norm
$\|M(x)\|$ is positive so the continuous map
\[
x \in \Dbar \to M(x) \to \|M(x)\|
\]
has a positive lower bound since $\Dbar$ is compact.

\noindent
{\underline{Proof of (a).}}
By hypothesis,
\[
\xi_{d+1} = \sum_{i=1}^d \alpha_i \xi_i
\]
for some $\alpha_i \geq 0$ with $\sum\limits_{i=1}^d \alpha_i =1$ and at least two of the 
$\alpha_i$ are non-zero. Hence, for any $x \in \Bbar$,
\beqn
x - \xi_{d+1} = \sum_{i=1}^d \alpha_i (x - \xi_i).
\label{eq:diverse1}
\eeqn
Regarding vectors as columns, using elementary column operations,
we have the determinant relations
\begin{align*}
( \det M(x)) \, \prod_{i=1}^{d+1} |x - \xi_i|  &=
\begin{vmatrix} 
 |x-\xi_1|& \cdots & |x-\xi_d| & |x- \xi_{d+1}|
 \\
 x - \xi_1 & \cdots & x- \xi_d & x - \xi_{d+1}
 \end{vmatrix}
 \\
& = \begin{vmatrix} 
 |x-\xi_1|& \cdots & |x-\xi_d| & \beta
 \\
 x - \xi_1 & \cdots & x- \xi_d & 0
 \end{vmatrix}
 \\
&  = \beta \begin{vmatrix} 
 x - \xi_1 & \cdots & x- \xi_d
 \end{vmatrix}.
\end{align*}
where $\beta = |x- \xi_{d+1}| - \sum\limits_{i=1}^d \alpha_i |x-\xi_i|.$

For $x \in \Dbar$, the vectors $x - \xi_1, \cdots, x- \xi_d$ are linearly independent because, if
$ \sum\limits_{i=1}^d \lambda_i (x-\xi_i) =0$, then
\[
\left (\sum_{i=1}^d \lambda_i \right ) x= \sum_{i=1}^d \lambda_i \xi_i.
\]
If $\sum\limits_{i=1}^d \lambda_i=0$ then $\sum\limits_{i=1}^d \lambda_i \xi_i =0$ which forces 
$\lambda_i=0$ from the linear independence of $\xi^1, \cdots, \xi^d$. If
$\sum\limits_{i=1}^d \lambda_i \neq 0$ then
\[
x = \sum_{i=1}^d \sigma_i \xi_i
\]
with $\sum\limits_{i=1}^d \sigma_i =1$ where $\sigma_i = \frac{\lambda_i}{\sum\limits_{i=1}^d \lambda_i}$.
This violates the hypothesis that $\Dbar$ does not intersect ${\mc{A}(\xi_1, \cdots, \xi_d)}$. Hence, for $x \in \Dbar$, the determinant 
$\begin{vmatrix}  x - \xi_1 & \cdots & x- \xi_d \end{vmatrix}$ is non-zero.

Next, for $x \in \Dbar$, from (\ref{eq:diverse1}) and the triangle inequality, we have
\begin{align*}
|x - \xi_{d+1}| & = \left | \sum_{i=1}^d \alpha_i (x - \xi_i) \right |
< \sum_{i=1}^d \alpha_i |x- \xi_i|
\end{align*}
because $\alpha_i \geq 0$, the $x-\xi_i$, $i=1, \cdots, d$ are not parallel
(because they are linearly independent as shown above) and at least two of the 
$\alpha_i (x- \xi_i)$ are non-zero. Hence $\beta \neq 0$. 

So combining the conclusions of the previous two paragraphs, we have $\det M(x) \neq 0$ for all $x \in \Dbar$, which completes the proof of (a).


\noindent
{\underline{Proof of (b).}}

We start with the claim that every $x \in \real^d$ has a unique representation as
$ x = \sum\limits_{i=1}^{d+1} \alpha_i \xi_i$
for some $\alpha_i \in \real$ with $\sum\limits_{i=1}^{d+1} \alpha_i=1$. We postpone the proof of this claim to the end of this section and continue with the proof of (b).


For $x \in \Dbar$, the invertibility of $M(x)$ is equivalent to the 
linear independence of the vectors 
\[
[|x-\xi_1|, x- \xi_1],  \cdots, [|x-\xi_{d+1}|, x- \xi_{d+1}]
\]
in $\real^{d+1}$. If there are 
$\lambda_1, \cdots, \lambda_{d+1} \in \real$ such that
\[
\sum_{i=1}^{d+1} \lambda_i [|x-\xi_i|, x-\xi_i] =0
\]
then
\begin{align}
\sum\limits_{i=1}^{d+1} \lambda_i |x-\xi_i|  =0, \qquad \sum_{i=1}^{d+1} \lambda_i (x - \xi_i) = 0.
\label{eq:conv1}
\end{align}

If $ \sum\limits_{i=1}^{d+1} \lambda_i \neq 0$, define 
\[
\mu_i= \frac{\lambda_i}{\sum\limits_{i=1}^{d+1} \lambda_i}, \qquad i=1, \cdots, d+1.
\]
Then $\sum\limits_{i=1}^{d+1} \mu_i = 1$ and (\ref{eq:conv1}) implies
\beqn
x = \sum_{i=1}^{d+1} \mu_i \xi_i, \qquad \sum_{i=1}^{d+1} \mu_i |x-\xi_i| =0.
\label{eq:conv2}
\eeqn
Now the $\mu_i$ are uniquely determined because of the claim in the second paragraph of the proof of (b). 
Further, since $x \in {\mc{C}(\xi_1, \cdots, \xi_{d+1})}$, we have $\mu_i \geq 0$, so the relation
$\sum\limits_{i=0}^n \mu_i = 1$ implies least one of the $\mu_i$ is positive. 
Hence the second
equation in (\ref{eq:conv2}) implies that $x= \xi_i$ for at least one of the $i$, which contradicts our assumption that any $x \in \Dbar$ is in the interior of ${\mc{C}(\xi_1, \cdots, \xi_{d+1})}$.

So we must have $\sum\limits_{i=1}^{d+1} \lambda_i =0$; then (\ref{eq:conv1}) implies 
$\sum\limits_{i=1}^{d+1} \lambda_i \xi_i = 0$. From our claim, every $x \in \real^d$ has a unique representation
\[
x = \sum_{i=1}^{d+1} \alpha_i \xi_i
\]
for some $\alpha_i$ with $\sum\limits_{i=1}^{d+1} \alpha_i =1$. However, we also have
\[
x = \sum\limits_{i=1}^{d+1} (\alpha_i + \lambda_i) \xi_i
\]
with $\sum\limits_{i=0}^{d+1} ( \alpha_i + \lambda_i) = 1$.
So the unique representation property implies $\lambda_i=0$, $i=1, \cdots, {d+1}$, proving (b). 

It remains to prove the unique representation claim stated at the beginning of the proof of (b). We
observe that if $\alpha_i \in \real$ with $\sum\limits_{i=1}^{d+1} \alpha_i = 1$ then
\begin{align}
\sum_{i=1}^{d+1} \alpha_i \xi_i & = \sum_{i=1}^d \alpha_i ( \xi_i - \xi_{d+1}) +
\left (\sum_{i=1}^{d+1} \alpha_i \right ) \xi_{d+1}
\nonumber
\\
& = \xi_{d+1} + \sum_{i=1}^d \alpha_i ( \xi_i - \xi_{d+1}).
\label{eq:ximxd}
\end{align}
Hence
\[
{\mc{A}(\xi_1, \cdots, \xi_{d+1})} = \xi_{d+1} + \text{span}(\xi_1 - \xi_{d+1}, \cdots, \xi_d - \xi_{d+1}).
\]

Now $D$ is an open subset of $\real^d$ contained in ${\mc{C}(\xi_1, \cdots, \xi_{d+1})}$ which is a subset
of ${\mc{A}(\xi_1, \cdots, \xi_{d+1})}$. Hence $\xi_1 - \xi_{d+1}, \cdots, \xi_d - \xi_{d+1}$ must be a basis
for $\real^d$. So $\text{span}(\xi_1 - \xi_{d+1}, \cdots, \xi_d - \xi_{d+1}) = \real^d$ and 
${\mc{A}(\xi_1, \cdots, \xi_{d+1})} = \real^d$. Finally, the representation is unique because
of (\ref{eq:ximxd}) and the linear independence of $\xi_1 - \xi_{d+1}, \cdots, \xi_d - \xi_{d+1}$.
\end{proof}

\section{Appendix}
In this section we prove the existence of a unique distributional solution
of an IVP for a second order hyperbolic PDE, along with an estimate of the solution in terms of the coefficients. This is a standard result but a statement and a proof of the result, suitable for our use, is difficult to find. We give a standard proof based on the well-posedness result for an 
IBVP for second order hyperbolic PDEs in \cite{Evansbook}.

We use the notation for time dependent Sobolev spaces in section 5.9.2 of \cite{Evansbook}. Suppose $T>0$, $D$ is a bounded region in $\real^n$ with a smooth boundary, $a(x,t), q(x,t)$ are compactly supported smooth functions on $ \real^n \times \real$ and $b(x,t)$ is a compactly supported smooth n-dimensional vector field on 
$\real^n \times \real$. Define
\[
\Lcal := \pa_t^2 - \Delta - a \pa_t + b \cdot \nabla + q, \qquad D_T = D \times (0,T),
\]
the $L^2$ inner product
\[
( v, w )  = \int_D v(x) \, w(x) \, dx,
\qquad v,w \in L^2(D),
\]
and the bilinear forms
\begin{align*}
A[v,w;t] & = - \int_D a(x,t) \, v(x) \, w(x) \, dx,
\\
B[v,w;t] & = \int_{D} \nabla v(x) \cdot \nabla w(x) + b(x,t) \cdot \nabla v(x) \, w(x) + q(x,t) \, v(x) \, w(x) \, dx,
\end{align*}
for $v, w \in H^1(D)$, $0 \leq t \leq T$. For functions $u(x,t)$ on $D \times (0,T)$, the expression
$u(t)$ will denote the function $u(t): D \to \real$ with $u(t)(x) = u(x,t)$.

For a function $F \in L^2(D_T)$, consider the IBVP 
\begin{align}
\Lcal u = F, & \qquad \text{on } D \times (0,T),
\label{eq:evansde}
\\
u(\cdot,t\smeq 0) = 0, \qquad u_t(\cdot, t \smeq 0) = 0, & \qquad \text{on } D,
\label{eq:evansic}
\\
u = 0 & \qquad \text{on } \pa D \times [0,T].
\label{eq:evansbc}
\end{align}
A function $u \in L^2(0,T; H_0^1(D))$ with $u_t \in L^2(0,T; L^2(D))$ and $u_{tt} \in L^2(0,T; H^{-1}(D))$ is said to be a weak solution of the IBVP \evansibvp if the following holds:
\begin{flalign}
(i)  &~~( u_{tt}(t), v) + A[u_t(t), v; t] + B[u(t), v; t] = (F(t), v), 
~~ \forall v \in H_0^1(D), ~~ \text{a.e. } 0 \leq t \leq T,&&
\label{eq:weakde}
\\
(ii) & ~~ u(\cdot, 0) = 0, ~~ u_t(\cdot, 0) = 0. &&
\label{eq:weakic}
\end{flalign}
Note that by Sobolev space theory, $u \in C([0,T], L^2(D))$ and $u_t \in C([0,T], H^{-1}(D))$, so (ii) makes sense.
We also observe that if the weak solution $u$ is in $H^2(D \times (0,T))$ then a standard argument shows that $u$ satisfies \eqref{eq:evansde} and \eqref{eq:evansbc} as functions.

 Theorems 3,4,5 in Section 7.2 in \cite{Evansbook} give a well-posedness result for this IBVP. Further, 
 Theorem 6 in Section 7.2 in \cite{Evansbook} gives higher order regularity if $F$ has higher order regularity and satisfies a matching condition on $\pa D \times [0,T]$. We need only a special case of the general results in \cite{Evansbook}.
\begin{prop}\label{prop:evansm}
 If $F \in L^2(D_T)$ then the IBVP \evansibvp
 has a unique weak solution $u$. Further
 $u \in L^\infty(0,T; H_0^1(D))$ and $u_t \in L^\infty(0,T; L^2(U))$ with
 \beqn
 \esssup_{0 \leq t \leq T}
  \left ( \|u(t)\|_{H_0^1(D)} + \|u_t(t)\|_{L^2(D)} \right )
 \leq C  \|F\|_{L^2(D_T)}
 \eeqn
 and $C$ determined by $T$ and $\|[a,b,q]\|_{L^\infty(D_T)}$. Further, if $m$ is a positive integer and
 \begin{flalign}
& \pa_t^k F \in L^2(0,T; H^{m-k}(D)) \qquad \text{for } k=0, \cdots,m, && 
\\
 & (\pa_t^k F)(\cdot,0)|_{\pa D} = 0 \qquad  \text{for $k=0, \cdots, m-2$}, &&
\end{flalign}
then $ \pa_t^k u \in L^\infty(0,T; H^{m+1-k}(D))$ for $k=0,1, \cdots, m+1$ and we have the estimate
\beqn
\esssup_{0 \leq t \leq T} \; \sum_{k=0}^{m+1} \| \pa_t^k u(\cdot,t)\|_{H^{m+1-k}(D)}
 \leq C \sum_{k=0}^m  \|\pa_t^k F\|_{L^2(0,T; H^{m-k}(D))},
 \eeqn
 with $C$ determined by $T$ and $\|[a,b,q]\|_{C^m(D_T)}$.
\end{prop}
The theorems in \cite{Evansbook} are for a general second order hyperbolic operator similar to our $\Lcal$ except
with $\Delta$ replaced by a general second order elliptic operator, and without the $a \pa_t$ term.
However, with minor modifications (particularly to the proof of Theorem 4 in section 7.2 of \cite{Evansbook}), 
the same proof works for our $\Lcal$. In our proposition, we have also added the dependence of $C$ on the coefficients, which follows easily if, in the proof, we track the dependence of the constants on the coefficients.

%

From Proposition \ref{prop:evansm} we derive the following existence result for an IVP, needed below. 
\begin{prop}\label{prop:evansivp}
Suppose $m$ is a positive integer, $\pa_t^k F \in L^2(0,T; H^{m-k}(\real^n))$ for  $k=0, \cdots,m$ and
$F$ is compactly supported. Then the IVP 
\begin{align}
\Lcal u = F, & \qquad \text{on } \real^n \times (0, T),
\label{eq:evansivpde}
\\
u(\cdot, t \smeq 0) = 0, ~~ u_t(\cdot, t \smeq 0) & \qquad \text{on } \real^n
\label{eq:evansivpic}
\end{align}
has a solution $u$ with $ \pa_t^k u \in L^\infty(0,T; H^{m+1-k}(\real^n))$ for $k=0,1, \cdots, m+1$. 
Further
\beqn
\esssup_{0 \leq t \leq T} \; \sum_{k=0}^{m+1} \| \pa_t^k u(\cdot,t)\|_{H^{m+1-k}(\real^n)}
 \leq C \sum_{k=0}^m  \|\pa_t^k F\|_{L^2(0,T; H^{m-k}(\real^n))},
 \eeqn
 with $C$ determined by $T$ and $\|[a,b,q]\|_{C^m(\real^n \times [0,T])}$.
\end{prop}

\begin{proof}
Suppose $F$ is supported in $B_R \times [0,T]$ where $B_R$ is the origin centered ball of radius $R$. Let $D$ be the origin centered ball of radius $2R+T$. Then $F$ satisfies the conditions of Proposition \ref{prop:evansm} so
the IBVP \eqref{eq:evansde} - \eqref{eq:evansbc} has a solution $u$ with $ \pa_t^k u \in L^\infty(0,T; H^{m+1-k}(D))$ for $k=0,1, \cdots, m+1$ and 
\beqn
\esssup_{0 \leq t \leq T} \; \sum_{k=0}^{m+1} \| \pa_t^k u(\cdot,t)\|_{H^{m+1-k}(D)}
 \leq C \sum_{k=0}^m  \|\pa_t^k F\|_{L^2(0,T; H^{m-k}(D))},
 \eeqn
 with $C$ determined by $T$ and $\|[a,b,q]\|_{C^m(D_T)}$.
 
 Since $m \geq 1$, we have $u \in H^2(D_T)$, so $ u_t^2 + |\nabla u|^2$ and $ u_t \nabla u$ are in 
 $W^{1,1}(D_T)$. Since the divergence theorem is valid, on regions with Lipschitz boundary, for vector fields with components in $W^{1,1}(D_T)$, using a standard energy estimate argument on a truncated cone and that
 $F=0$ outside $B_R \times [0,T]$, one can show that $u(x,t)=0$ for $R+T \leq |x| \leq 2R+T$, $0 \leq t \leq T$. Hence if we define $u=0$ for $|x| \geq 2R+T$, $0 \leq t \leq T$, then we have a solution of the IVP \eqref{eq:evansivpde}, \eqref{eq:evansivpic} with the properties claimed in Proposition \ref{prop:evansivp}. 
 \end{proof}
 
 Suppose $F$ is a distribution on $\real^n \times (-\infty, T)$ with $F=0$ for $t<0$. Consider the IVP
 \begin{align}
\Lcal u = F, & \qquad \text{on } \real^n \times (-\infty, T),
\label{eq:infde}
\\
u=0 & \qquad \text{on } \real^n \times (-\infty, 0).
\label{eq:infic}
\end{align}
We say a distribution $u$ on $\real^n \times (-\infty,T)$ is a solution of this IVP if $u =0$ for $t<0$ and
\[
\la u, \Lcal^* \phi \ra =\la F, \phi \ra, \qquad \forall \phi \in C_c^\infty(\real^n \times (-\infty, T) );
\]
here $\Lcal^*$ is the formal adjoint of $\Lcal$. We have the following well-posedness result for the IVP \eqref{eq:infde},
\eqref{eq:infic}.
\begin{prop}\label{prop:infivp}
Suppose $m$ is a positive integer, $\pa_t^k F \in L^2(-\infty, T; H^{m-k}(\real^n))$ for  $k=0, \cdots,m$, 
$F$ compactly supported and $F=0$ for $t<0$. Then the IVP \eqref{eq:infde}, \eqref{eq:infic} has a unique distributional solution $u$. Further, if $m>(n-1)/2$ then for any non-negative integer $p < m - (n-1)/2$ we have 
$u \in C^p(\real^n \times (-\infty, T])$ and
\[
\| u \|_{C^p(\real^n \times (-\infty, T])}
\leq C \sum_{k=0}^m  \|\pa_t^k F\|_{L^2(0,T; H^{m-k}(\real^n))},
 \]
 with $C$ determined by $T$ and $\|[a,b,q]\|_{C^m(\real^n \times [0,T])}$.
 \end{prop}
\begin{proof}
If we apply Proposition \ref{prop:evansivp} with the initial condition
$u (\cdot, -\ep)=0, ~ u_t(\cdot, - \ep)=0$ for some $\ep>0$, then we are guaranteed a solution 
$u \in H^{m+1}(\real^n \times (-\ep, T))$ with
\[
\| u \|_{H^{m+1}(\real^n \times (-\ep, T) )} \leq C  \sum_{k=0}^m  \|\pa_t^k F\|_{L^2(-\ep,T; H^{m-k}(\real^n))},
\]
 with $C$ determined by $T$ and $\|[a,b,q]\|_{C^m(\real^n \times [-\ep,T])}$.
Since $m \geq 1$, we have $u \in H^2$, so by an energy estimate $u$ will be zero for $-\ep <t<0$. We extend $u$ as the zero function for the region $t \leq - \ep$; then $u$ is a distributional solution of \eqref{eq:infde}, \eqref{eq:infic}.

Suppose $m>(n-1)/2$. Noting that $u$ is compactly supported with the support determined by $T$ and the support of $F$, from the Sobolev embedding theorem, for any non-negative integer 
$p < m+1 - (n+1)/2 = m - (n-1)/2$, we have $u \in C^p(\real^n \times (-\infty, T])$ and
\[
\| u \|_{C^p(\real^n \times (-\infty, T])}
\leq C \sum_{k=0}^m  \|\pa_t^k F\|_{L^2(0,T; H^{m-k}(\real^n))},
 \]
 with $C$ determined by $T$ and $\|[a,b,q]\|_{C^m(\real^n \times [0,T])}$. 
It remains to prove the uniqueness of the distributional solution. 

Note that for $F$ regular enough, there is a $C^2$ solution of \eqref{eq:infde}, \eqref{eq:infic}. Further, a standard energy estimate shows that there is at most one $C^2$ solution. This will be important for us in our proof next of the claim that if $F=0$ then any distributional solution $u$ of \eqref{eq:infde}, \eqref{eq:infic} must be zero. 

Suppose $\phi$ is a compactly supported smooth function on $\real^n \times \real$ with support $0 \leq t \leq T$. 
Consider the backward IVP
\begin{align*}
\Lcal^* v = \phi & \qquad \real^n \times (-3,\infty)
\\
v = 0 & \qquad \text{on } t>T.
\end{align*}
Then, reversing time $t$ and using the existence part (for arbitrary large $m$) and the uniqueness of $C^2$ solutions (proved already), we know there is a smooth solution $v$ on $\real^n \times (-3, \infty)$ of this backward IVP.
Further, the restriction of $v$ to $ t \geq t_1$ is compactly supported for any $t_1>-3$.

Let $\rho(t)$ be a smooth function on $\real$ with 
\[
\rho(t) = \begin{cases} 1, & t > -1 \\ 0, & t <-2 \end{cases},
\]
and define $w(x,t) = \rho(t) v(x,t)$. Then $w$ is a compactly supported smooth function on $\real^n \times \real$ and, on the region $t>-1$ we have $\Lcal^*( w) = \Lcal^* v = \phi$. Noting that $u=0$ for $t<0$ and using the definition of a distributional solution, we have
\[
\la u, \phi \ra = \la u, \Lcal^* w \ra = 0, \qquad \forall \phi \in C_c^\infty (\real^n \times (-\infty, T) ).
\]
Hence $u=0$ on $\real^n \times (-\infty, T)$.
\end{proof}


\section{Acknowledgements}	
Rakesh's work was supported by the NSF grant DMS 1908391. Soumen Senapati gratefully acknowledges the support and hospitality received during his stay at University of Delaware in 2019
where part of this work was done.


\end{document}